\documentclass[a4paper,10pt]{amsart}
\usepackage[utf8]{inputenc}
\usepackage{hyperref}
\usepackage{relsize}
\usepackage{amssymb,amsmath}
\usepackage{dsfont}
\usepackage[all]{xy}

\usepackage{stmaryrd}
\usepackage{upgreek}

\usepackage[lmargin=1.2in,rmargin=1.2in,tmargin=1.3in,bmargin=1.1in]{geometry}

\usepackage{color}

\usepackage[utf8]{inputenc}
\usepackage{hyperref}
\usepackage{relsize}
\usepackage{amssymb,amsmath}
\usepackage{dsfont}
\usepackage[all]{xy}
\usepackage{mathtools}
\usepackage{mathrsfs}

\renewcommand{\theta}{\uptheta}
\renewcommand{\iota}{\upiota}
\renewcommand{\alpha}{\upalpha}
\renewcommand{\beta}{\upbeta}
\renewcommand{\gamma}{\upgamma}
\renewcommand{\delta}{\updelta}
\renewcommand{\zeta}{\upzeta}
\renewcommand{\pi}{\uppi\hspace{0.05em}}
\renewcommand{\xi}{\upxi}
\renewcommand{\chi}{\upchi}
\renewcommand{\sigma}{\upsigma}
\renewcommand{\Lambda}{\Uplambda}
\renewcommand{\Gamma}{\Upgamma}
\renewcommand{\phi}{\upphi}
\renewcommand{\nu}{\upnu}
\renewcommand{\tau}{\uptau}
\renewcommand{\mu}{\upmu}
\renewcommand{\eta}{\upeta}

\newtheorem{theorem}{Theorem}[section]
\newtheorem{thmx}{Theorem}

\newtheorem{proposition}[theorem]{Proposition}
\newtheorem{lemma}[theorem]{Lemma}

\newtheorem{corollary}[theorem]{Corollary}

\theoremstyle{remark}
\newtheorem{remark}[theorem]{Remark}

\newtheorem{example}[theorem]{Example}
\newtheorem{question}[theorem]{Question}

\theoremstyle{definition}

\newcommand{\dvs}{\mathbb{N}^{Q_0}}

\renewcommand{\AA}{\mathbb{A}}

\newcommand{\CC}{\mathbb{C}}

\newcommand{\LLL}{\mathscr{L}}

\newcommand{\kac}{\mathtt{a}}
\newcommand{\dd}{\mathbf{d}}

\newcommand{\ee}{\mathbf{e}}

\newcommand{\JH}{\mathtt{JH}}

\newcommand{\ff}{\mathbf{f}}

\DeclareMathOperator{\opp}{op}

\newcommand{\Mst}{\mathfrak{M}}

\newcommand{\Msp}{\mathcal{M}}

\newcommand{\ICS}{\mathcal{IC}}

\newcommand{\DTS}{\mathcal{BPS}}

\newcommand{\phim}[1]{\phi^{\mon}_{#1}}

\newcommand{\Up}{{}^{\mathfrak{p}}\!}
\newcommand{\phin}[1]{\Up\phi_{#1}}
\newcommand{\psin}[1]{\Up\psi_{#1}}

\DeclareMathOperator{\Flag}{\mathcal{F}l}

\DeclareMathOperator{\stab}{st}

\DeclareMathOperator{\Even}{even}
\DeclareMathOperator{\Odd}{odd}

\DeclareMathOperator{\hdg}{hdg}
\DeclareMathOperator{\crit}{crit}

\DeclareMathOperator{\Hom}{Hom}

\DeclareMathOperator{\mon}{mon}

\DeclareMathOperator{\simp}{simp}

\DeclareMathOperator{\MMHM}{\mathbf{MMHM}}

\DeclareMathOperator{\Id}{Id}
\DeclareMathOperator{\rat}{\mathbf{rat}}

\DeclareMathOperator{\lMod}{-Mod}

\DeclareMathOperator{\supp}{supp}
\DeclareMathOperator{\con}{con}

\DeclareMathOperator{\Perv}{\mathbf{Perv}}
\DeclareMathOperator{\MHM}{\mathbf{MHM}}

\DeclareMathOperator{\DT}{BPS}
\DeclareMathOperator{\Sym}{\mathbf{Sym}}

\DeclareMathOperator{\op}{op}
\DeclareMathOperator{\Spec}{Spec}
\DeclareMathOperator{\Gl}{GL}

\DeclareMathOperator{\gl}{\mathfrak{gl}}
\DeclareMathOperator{\Sl}{SL}

\DeclareMathOperator{\id}{id}
\DeclareMathOperator{\Jac}{Jac}

\DeclareMathOperator{\Tr}{Tr}

\DeclareMathOperator{\pt}{pt}
\DeclareMathOperator{\Tot}{Tot}

\DeclareMathOperator{\cone}{cone}

\DeclareMathOperator{\vir}{vir}
\DeclareMathOperator{\Ob}{Ob}
\DeclareMathOperator{\Ho}{\mathcal{H}}
\DeclareMathOperator{\HO}{\mathbf{H}}

\DeclareMathOperator{\Coha}{\mathcal{A}}

\DeclareMathOperator{\rCoha}{\mathcal{RA}}

\DeclareMathOperator{\TS}{\mathtt{TS}}

\DeclareMathOperator{\Coh}{Coh}
\DeclareMathOperator{\KW}{\mathtt{KW}}

\DeclareMathOperator{\BPS}{BPS}

\DeclareMathOperator{\poin}{\mathtt{p}}
\newcommand{\Db}{\mathcal{D}^{\mathrm{b}}}
\newcommand{\Dbfg}{\mathcal{D}^{\mathrm{b}}_{\mathrm{fg}}}

\title[A boson-fermion correspondence in cohomological Donaldson--Thomas theory]{A boson-fermion correspondence in cohomological Donaldson--Thomas theory}
\author{Ben Davison}
\address{School of Mathematics\\
University of Edinburgh\\
Edinburgh EH9 3FD\\
UK}
\email{Ben.Davison{\char'100}ed.ac.uk }

\begin{document}

\begin{abstract}
We introduce and study a fermionization procedure for the cohomological Hall algebra $\mathcal{H}_{\Pi_Q}$ of representations of a preprojective algebra, that selectively switches the cohomological parity of the BPS Lie algebra from even to odd.  We do so by determining the cohomological Donaldson--Thomas invariants of central extensions of preprojective algebras studied in the work of Etingof and Rains, via deformed dimensional reduction.  
\smallbreak
Via the same techniques, we determine the Borel--Moore homology of the stack of representations of the $\mu$-deformed preprojective algebra introduced by Crawley--Boevey and Holland, for all dimension vectors.  This provides a common generalisation of the results of Crawley-Boevey and Van den Bergh on the cohomology of smooth moduli schemes of representations of deformed preprojective algebras, and my earlier results on the Borel--Moore homology of the stack of representations of the undeformed preprojective algebra.
\end{abstract}
\maketitle

\setcounter{tocdepth}{1}
\tableofcontents  

\section{Introduction}
Given a quiver $Q$, in \cite{KS2} Kontsevich and Soibelman define the cohomological Hall algebra 
\begin{equation}
\label{hwg}
\Coha_Q=\bigoplus_{\dd\in\dvs}\HO(\Mst_{\dd}(Q),\mathbb{Q})[-\chi_Q(\dd,\dd)],
\end{equation}
which has as underlying vector space the singular cohomology of the stack of finite-dimensional complex representations of $Q$, shifted in cohomological degree by the Euler form (see \eqref{chi_def} for the definition).  As indicated by \eqref{hwg}, this algebra is graded by the dimension vectors of representations of $Q$.  The multiplication is defined by taking push-forward and pull-back of cohomology in the usual correspondence diagram
\begin{equation}
\label{corr_diag}
\xymatrix{
\mathfrak{M}(Q)\times\mathfrak{M}(Q)&&\ar[ll]_-{\pi_1\times\pi_3}\mathfrak{E}\mathrm{xact}(Q)\ar[rr]^-{\pi_2}&&\mathfrak{M}(Q).
}
\end{equation}
In the diagram \eqref{corr_diag}, $\mathfrak{M}(Q)=\coprod_{\dd\in \dvs}\Mst_{\dd}(Q)$ is the stack of finite dimensional $Q$-representations, $\mathfrak{E}\mathrm{xact}(Q)$ is the stack of short exact sequences of $Q$-representations, and $\pi_i$ is the morphism taking a short exact sequence to its $i$th entry.
\smallbreak
If $Q$ is moreover symmetric (i.e. for every pair of vertices $i,j$, there are as many arrows from $i$ to $j$ as from $j$ to $i$), then the multiplication respects the cohomological degree, and a theorem of Efimov \cite{Efimov} states that (a slight modification of) $\Coha_Q$ is a free \textit{super}commutative algebra; the ``super'' here means that elements in odd cohomological degrees anti-commute with each other.  
\smallbreak
There are a couple of instances in which this result is possible to check by hand, indeed it was observed in \cite[Sec.2.5]{KS2} that if $Q^{(l)}$ denotes the quiver with one vertex and $l$ loops, then
\begin{align}
\Coha_{Q^{(0)}}\cong &\Sym\left((\mathbb{Q}[u])[-1]\right)\label{se}\\
\Coha_{Q^{(1)}}\cong &\Sym\left(\mathbb{Q}[u]\right)\label{sc}
\end{align}
where $[-1]$ denotes a cohomological shift, and $u^i$ is placed in cohomological degree $2i$.  Both algebras are free supercommutative algebras generated by a countable set of symbols $\alpha_i$ for $i\geq 0$, with $\alpha_i$ placed in cohomological degree $1+2i$ in the zero-loop case, and in cohomological degree $2i$ in the one-loop case.  In other words, \eqref{se} states that $\Coha_{Q^{(0)}}$ is a free exterior algebra with countably many generators, while \eqref{sc} states that $\Coha_{Q^{(1)}}$ is a free commutative algebra with the same generators.  As observed in \cite{KS2}, the fact that the underlying vector spaces of the two algebras are the same (since $\Mst_n(Q^{(l)})$ is homotopic to $\mathrm{BGl}_n$, regardless of $l$) can be seen as a consequence of the boson-fermion correspondence in representation theory.
\smallbreak
Now let $Q$ be an arbitrary finite quiver.  We form the tripled quiver $\tilde{Q}$, by adjoining to the quiver $Q$ an arrow $a^*$ with the opposite orientation to $a$, for each $a$ an arrow of $Q$, and also adjoining a loop $\omega_i$ at each vertex $i$ of $Q$.  So for instance $\widetilde{Q^{(1)}}\cong Q^{(3)}$.  The quiver $\tilde{Q}$ carries a canonical cubic potential 
\begin{equation}
\label{wt_def}
\tilde{W}=\sum_{a\in Q_1}[a,a^*]\sum_{i\in Q_0} \omega_i,
\end{equation}
and one may define (again as in \cite{KS2}) the critical cohomological Hall algebra $\Coha_{\tilde{Q},\tilde{W}}$.  Again, the multiplication respects cohomological degree.  The underlying vector space of this algebra is the vanishing cycle cohomology of the function $\Tr(\tilde{W})$ on the stack $\mathfrak{M}(\tilde{Q})$.  The multiplication is defined via pull-back and push-forward of vanishing cycle cohomology along the same correspondence diagram \eqref{corr_diag}.  Via dimensional reduction \cite{Da13} the algebra $\Coha_{\tilde{Q},\tilde{W}}$ is isomorphic \cite{RS17,YZ16} to the cohomological Hall algebra structure on the Borel--Moore homology of the stack of representations of the preprojective algebra $\Pi_Q$ constructed by Schiffmann and Vasserot in \cite{ScVa13} and studied in \cite{YZ18}.
\smallbreak
In general the algebra $\Coha_{\tilde{Q},\tilde{W}}$ is not supercommutative; for example $\Coha_{\tilde{Q},\tilde{W}}$ contains the universal enveloping algebra of the Kac--Moody Lie algebra associated to the quiver $Q'$ obtained by removing all vertices from $Q$ that support 1-cycles \cite{preproj3}.  On the other hand, via the cohomological integrality theorem \cite{QEAs}, the entire algebra $\Coha_{\tilde{Q},\tilde{W}}$ satisfies a Poincar\'e--Birkhoff--Witt theorem, meaning that we have an $\dvs$-graded isomorphism of cohomologically graded vector spaces (but \textit{not} of algebras):
\begin{align}
\label{syma}
\Coha_{\tilde{Q},\tilde{W}}\cong &\Sym\left(\bigoplus_{\dd\in\dvs}\left(\DT_{\tilde{Q},\tilde{W},\dd}\otimes \mathbb{Q}[u]\right)[-1]\right).
\end{align}
In the isomorphism \eqref{syma}, $\DT_{\tilde{Q},\tilde{W},\dd}$ is a cohomologically graded vector space, the whole of $\DT_{\tilde{Q},\tilde{W},\dd}\otimes \mathbb{Q}[u]$ is placed in $\dvs$-degree $\dd$, $u^i$ has cohomological degree $2i$, and the symbol $[-1]$ denotes the cohomological shift as before.  Moreover by \cite{preproj}, the Poincar\'e polynomial of the cohomologically graded vector space $\DT_{\tilde{Q},\tilde{W},\dd}$ satisfies the relation
\begin{equation}
\label{KactoDT}
\poin(\DT_{\tilde{Q},\tilde{W},\dd},q^{1/2})\coloneqq \sum_{i\in\mathbb{Z}}\dim(\DT_{\tilde{Q},\tilde{W},\dd}^i)q^{i/2}=q^{-1/2}\kac_{Q,\dd}(q^{-1})
\end{equation}
where $\kac_{Q,\dd}(q)$ is the \textit{Kac polynomial} \cite{Kac80}, counting absolutely indecomposable $\dd$-dimensional $Q$-representations over a field of order $q$.  For example, we observe that $Q^{(1)}\cong\widetilde{Q^{(0)}}$ and calculate 
\[
\kac_{Q^{(0)},\dd}(q)=\begin{cases} 1&\textrm{if }\dd=1\\0&\textrm{if }\dd\geq 2\end{cases}
\]
recovering \eqref{sc} from \eqref{syma} and \eqref{KactoDT}.
\smallbreak
In particular, we see that for arbitrary finite $Q$, the algebra $\Coha_{\tilde{Q},\tilde{W}}$ is \textit{bosonic}, in the sense that it is situated entirely in even cohomological degree.  Having observed that $\Coha_{Q^{(1)}}=\Coha_{\widetilde{Q^{(0)}},\tilde{W}}$ has a fermionic counterpart $\Coha_{Q^{(0)}}$, we may ask the following
\begin{question}
\label{bq}
For a finite quiver $Q$, is there a ``fermionic'' version of $\Coha_{\tilde{Q},\tilde{W}}$?
\end{question}
\smallbreak
It is a consequence of the \textit{purity} of the cohomological BPS invariants for the quiver with potential $(\tilde{Q},\tilde{W})$ (again, see \cite{preproj}) that there is an equality
\begin{equation}
\label{ppv}
\poin(\DT_{\tilde{Q},\tilde{W},\dd},q^{1/2})=\chi_{q^{1/2}}(\DT_{\tilde{Q},\tilde{W},\dd})\
\end{equation}
between the Poincar\'e polynomial and the virtual Poincar\'e polynomial\footnote{See \cite{Moz11} for the virtual Poincar\'e polynomial analogue of \eqref{KactoDT}; any two of purity, \eqref{KactoDT}, and the virtual analogue of \eqref{KactoDT} imply the third.} of $\DT_{\tilde{Q},\tilde{W},\dd}$.  The polynomials $\chi_{q^{1/2}}(\DT_{\tilde{Q},\tilde{W},\dd})$ are essentially by definition the \textit{refined} BPS invariants for the quiver $\tilde{Q}$ with potential $\tilde{W}$.  Expressed in the language of refined Donaldson--Thomas theory, the analogue of Question \ref{bq} is
\begin{question}
\label{bq2}
Is there quiver $Q'$ with potential $W'$ such that the refined BPS invariants satisfy
\[
\chi_{q^{1/2}}(\DT_{Q',W',\dd})=q^{1/2}\chi_{q^{1/2}}(\DT_{\tilde{Q},\tilde{W},\dd})?
\]
\end{question}
In this paper we will answer these two questions in the affirmative.
\subsection{Counting rational curves}
\label{cs_sec}
At least in the case in which $Q$ is an affine Dynkin quiver, there are strong hints from the McKay correspondence that the answer to Question \ref{bq2} should be at least a partial ``yes''.  We recall some geometric background; see \cite{KM92,BKL01} for more details.
\smallbreak
Fix a finite subgroup $G\subset \Sl_2(\mathbb{C})$, and denote by $X_0=\mathbb{C}^2/G$ the associated Kleinian singularity.  We denote by $\Gamma$ the McKay graph of $G$, and by $\Gamma'$ the full ADE type sub-graph of $\Gamma$ obtained by removing the vertex corresponding to the trivial representation.  We denote by $p\colon Y_0\rightarrow X_0$ the minimal resolution of $X_0$.  The surface $X_0$ contains an isolated singularity $x$, and the exceptional fibre $p^{-1}(x)$ consists of a chain of rational curves, with incidence graph $\Gamma'$.  The space $Y_0$ has a universal deformation $Y$ parametrised by $\mathfrak{h}$, the Cartan subalgebra of the simple Lie algebra corresponding to $\Gamma'$, so that we have a Cartesian diagram
\[
\xymatrix{
Y_0\ar[d]\ar@{^{(}->}[r]&Y\ar[d]^{\pi}\\
0\ar@{^{(}->}[r]&\mathfrak{h}.
}
\]
The generic fibre of $\pi$ contains no rational curves, while if $h\in\mathfrak{h}$ lies in a root hyperplane corresponding to a vertex $i$ of $\Gamma'$, the rational curve $C_i$ corresponding to $i$ deforms along the line $t\cdot h$ for $t\in\mathbb{C}$ \cite[Prop.2.2]{BKL01}.  We pick $h\in\mathfrak{h}$ and form the Cartesian diagram
\begin{equation}
\label{1pa}
\xymatrix{
Y^h\ar[d]^{\pi'}\ar[r]&Y\ar[d]^{\pi}\\
\mathbb{A}^1\ar[r]^-{t\mapsto t\cdot h}&\mathfrak{h}.
}
\end{equation}
\smallbreak
For $\beta\in \HO_2(Y^h,\mathbb{Z})$ and $n\in\mathbb{N}$, let $\mathcal{M}_{\beta,n}(Y^h)$ denote the moduli space of semistable coherent sheaves $\mathcal{F}$ on $Y'$, with fundamental class of the support of $\mathcal{F}$ equal to $\beta$, and with $\chi(\mathcal{F})=n$.  
\smallbreak 
Consider the case in which $\beta=[C_i]$ for some $i\in Q_0$, so that stability is equivalent to semistability.  Since a stable coherent sheaf $\mathcal{F}$ cannot split as a direct sum, it is supported on a single fibre of the morphism $\pi'$.  If $C_i$ deforms along $\mathbb{A}^1$, with curve over $t\in\mathbb{A}_1$ labelled $C_{i,t}$, then in each fibre $\pi'(t)$ there is a unique semistable coherent sheaf $\mathscr{O}_{C_{i,t}}(n-1)$ with Euler characteristic $n$ and $[\mathscr{O}_{C_{i,t}}(n-1)]=\beta$.  If $C_i$ does not deform, then there is a unique semistable coherent sheaf on the whole of $Y^h$ of class $(\beta,n)$, supported above $0\in \mathbb{A}^1$.  So
\[
\mathcal{M}_{[C_i],n}(Y^h)=\begin{cases} \pt &\textrm{if }h\notin i^{\perp}\\
\mathbb{A}^1&\textrm{if }h\in i^{\perp}.\end{cases}
\]
The \textit{cohomological BPS invariants} $\DT_{\beta,n}$ of the 3-fold $Y^h$ are in general hard to define rigorously, involving vanishing cycle cohomology, d-critical structures \cite{Jo15}, orientation data, perverse filtrations, etc.  However, for simple classes like the ones we are considering here, the definition/calculation boils down to something more straightforward:
\begin{align*}
\DT_{[C_i],n}(Y^h)=&\HO(\mathcal{M}_{[C_i],n}(Y^h),\mathbb{Q})_{\vir}\\
\cong&\begin{cases}\mathbb{Q}&\textrm{if }h\notin i^{\perp}\\
\mathbb{Q}[1]&\textrm{if }h\in i^{\perp}.
\end{cases}
\end{align*}
The subscript $\vir$ denotes the cohomological shift $[\dim(\mathcal{M}_{[C_i],n})]$ by the dimension of the space we are taking the singular cohomology of, and accounts for the shift in the $h\in i^{\perp}$ case.  So in both cases, the cohomological BPS invariants for simple curve classes are one-dimensional vector spaces, and are concentrated in even or odd cohomological degree, depending on whether the choice of $h$ means that $C_i$ is rigid or not.
\smallbreak
Similarly, the definition and calculation of the simplest degree zero cohomological BPS invariant $\DT_{0,1}(Y^h)$ is much easier than the general case, and we have
\[
\DT_{0,1}(Y^h)\cong \HO(Y^h,\mathbb{Q})_{\vir}.
\]
It is easy to verify that the singular cohomology of $Y^h$ does not depend on the choice of $h$ at all.  In general one expects $\DT_{0,n}\cong \DT_{0,1}$ (compare with \cite{BBS}), so it turns out that the degree zero\footnote{By which we mean the Donaldson--Thomas theory of coherent sheaves with zero-dimensional support.} cohomological DT theory of $Y^h$ does not depend on $h$.
\smallbreak
To put this geometric discussion in very leading language: for the most degenerate case $h=0$, all of the cohomological DT theory is bosonic, since the vanishing cycle cohomology ends up living in even cohomological degree (taking into account the shift $[-1]$ in \eqref{syma}).  On the other hand, modifying the deforming family defined by $h\in\mathfrak{h}$ to be more generic, a portion of the cohomological DT theory is fermionized, depending on which root hyperplanes $h$ avoids.
\subsection{The noncommutative conifold and central extensions of the preprojective algebra}
\label{ERC_sec}
Let $Q'$ be the $A_1$ quiver, and let $Q$ be its affine extension, which we label as follows:
\begin{equation}
\label{Qdef}
\xymatrix{
Q=&0\ar@/^1pc/[rr]^a&&1.\ar@/^1pc/[ll]^b
}
\end{equation}
Then as a special case of \cite{KV00} there is a derived equivalence between the category of finitely generated modules for the preprojective algebra $\Pi_Q$ of $Q$, and the category of coherent sheaves on $Y_0$, the minimal resolution of the type $A_1$ singularity defined by the equation $x^2+y^2=z^2$:
\[
\Dbfg(\Pi_Q\lMod)\xrightarrow{\cong}\Db(\Coh(Y_0)).
\]
We have that 
\[
\xymatrix{
\tilde{Q}=&&0\ar@(ul,dl)_{\omega_0}\ar@/^1.5pc/[rr]^{b^*}\ar@/^1pc/[rr]_a&&1\ar@/^1pc/[ll]_b\ar@/^1.5pc/[ll]^{a^*}\ar@(ur,dr)^{\omega_1}
}
\]

The Kac polynomials of $Q$ are possible to calculate by hand, so that we can calculate the cohomological BPS invariants for $\Jac(\tilde{Q},\tilde{W})$ via \eqref{KactoDT}
\begin{equation}
\label{h0DT}
\DT_{\tilde{Q},\tilde{W},(m,n)}\cong\begin{cases} \mathbb{Q}[3]\oplus\mathbb{Q}[1]&\textrm{if }m=n\\
\mathbb{Q}[1] &\textrm{if }m=n\pm 1\\
0&\textrm{otherwise.}
\end{cases}
\end{equation}

There is an isomorphism $\Jac(\tilde{Q},\tilde{W})\cong \Pi_Q[\omega]=\Pi_Q\otimes \mathbb{C}[\omega]$, so that we have in addition a derived equivalence (see e.g. \cite{Sz08})
\[
\Dbfg(\Jac(\tilde{Q},\tilde{W})\lMod)\xrightarrow{\cong}\Db(\Coh(Y^0))
\]
where $Y^0$ is defined via the construction in the previous subsection by setting $h=0\in\mathfrak{h}$, i.e. $Y^0=Y_0\times \mathbb{A}^1$.
\smallbreak
For the $A_1$-singularity, the Cartan sub-algebra $\mathfrak{h}$ is one-dimensional, so aside from $0$ there is an essentially unique choice of $h\in\mathfrak{h}$.  Defining $Y_{\con}=Y^h$ for a \textit{nonzero} choice of $h\in\mathfrak{h}$ in diagram \eqref{1pa}, we obtain the resolved conifold.  As noted in the previous section, instead of having an $\mathbb{A}^1$-family of rational curves giving rise to cohomological BPS invariants in odd degrees, the resolved conifold contains a unique rigid curve.  It follows that the cohomological BPS invariants corresponding to sheaves supported on the curve flip parity, and are supported in even cohomological degrees, so that they contribute to the fermionic part of the cohomological DT theory of the resolved conifold (as ever, taking into account the shift defined as in \eqref{syma})
\smallbreak
The resolved conifold also has a noncommutative model, studied in this context by Szendr\H{o}i, which we recall (see \cite{conifold} for details on the noncommutative Donaldson--Thomas theory of the conifold, and also \cite{RSYZ20,GY20,LY20} for more recent work on CoHAs related to toric 3-folds).  We consider the double
\[
\xymatrix{
\overline{Q}=&0\ar@/^1.5pc/[rr]^{b^*}\ar@/^1pc/[rr]_a&&1.\ar@/^1pc/[ll]_b\ar@/^1.5pc/[ll]^{a^*}
}
\]
Set $W_{\KW}=aa^*bb^*-a^*ab^*b$ to be the Klebanov--Witten potential \cite{KW98}.  Then (e.g. as a special case of a result due to Van den Bergh \cite{NonComCrep}) there is a derived equivalence
\[
\Dbfg(\Jac(\overline{Q},W_{\KW})\lMod)\xrightarrow{\cong}\Db(\Coh(Y_{\con})).
\]
To interpolate between the two cases ($h=0,h\neq 0$) it turns out to be more instructive to consider the quiver $\tilde {Q}$ with the potential 
\[
\tilde{W}^{(1,-1)}=\tilde{W}+\omega_0^2-\omega_1^2.
\]
As we recall in \S \ref{JAS} the resulting Jacobi algebra has already been studied: it is a special case of the central extensions of $\Pi_Q$ introduced by Etingof and Rains in \cite{ER05}.  There is an isomorphism (see Example \ref{ConEx})
\begin{equation}
\label{twoJac}
\Jac(\tilde{Q},\tilde{W}^{(1,-1)})\cong \Jac(\overline{Q},W_{\KW})
\end{equation}
and the cohomological DT theory of the Jacobi algebras in \eqref{twoJac} turns out to be the same.  The cohomological BPS invariants for the noncommutative conifold can be deduced from \cite{MMNS} and purity (proved as in \cite[Thm.4.7]{DOS20}):
\begin{equation}
\label{h1DT}
\DT_{\tilde{Q},\tilde{W}^{(1,-1)},m,n}\cong\begin{cases} \HO(Y_{\con},\mathbb{Q})[3]\cong \mathbb{Q}[3]\oplus \mathbb{Q}[1] &\textrm{if }m=n\\
\mathbb{Q}&\textrm{if }m=n\pm 1\\
0&\textrm{otherwise.}
\end{cases}
\end{equation}
Comparing \eqref{h0DT} with \eqref{h1DT} we see exactly the same pattern as in the commutative case of \S \ref{cs_sec}, with the passage from trivial $\mathbb{A}^1$-deformations of $Y_0$ to nontrivial ones replaced in the context of noncommutative algebraic geometry by the passage from the trivial central extension of the algebra $\Pi_Q$ to the nontrivial ones constructed by Etingof and Rains; this change provokes a flip in the parity of some, but not all, of the cohomological BPS invariants.

\subsection{Main results}
In the rest of the paper we prove that the above discussion regarding the (noncommutative) conifold forms part of a general procedure for (selectively) ``fermionizing'' the cohomological Hall algebras of preprojective algebras.  This culminates in the following theorem:
\begin{thmx}
\label{main_thm}
Let $\mu\in\mathbb{C}^{Q_0}$.  Set $\tilde{W}^{\mu}=\sum_{a\in Q_1}[a,a^*]\sum_{i\in Q_0}\omega_i+\frac{1}{2}\sum_{i\in Q_0} \mu_i\omega_i^2$.  Then the cohomological BPS invariants for the quiver $\tilde{Q}$ with potential $\tilde{W}^{\mu}$ satisfy
\begin{equation}
\label{refDTp}
\poin(\DT_{\tilde{Q},\tilde{W}^{\mu},\dd},q^{1/2})=\begin{cases}\kac_{Q,\dd}(q^{-1}) & \textrm{if } \dd\cdot \mu\neq 0\\
q^{-1/2}\kac_{Q,\dd}(q^{-1}) &\textrm{if }\dd\cdot \mu=0
\end{cases}
\end{equation}
where $\kac_{Q,\dd}(q)$ are the Kac polynomials for $Q$.  In addition, the natural mixed Hodge structure on the cohomological BPS invariants $\DT_{\tilde{Q},\tilde{W}^{\mu},\dd}$ is pure, of Tate type, so that we have isomorphisms of Hodge-theoretic BPS invariants
\begin{equation}
\label{cohDTa}
\DT^{\hdg}_{\tilde{Q},\tilde{W}^{\mu},\dd}\cong\begin{cases}\bigoplus_{i\in\mathbb{Z}}(\LLL^{i})^{\oplus \kac_{Q,\dd,-i}}&\textrm{if } \dd\cdot\mu\neq 0\\
\bigoplus_{i\in\mathbb{Z}}(\LLL^{i-1/2})^{\oplus \kac_{Q,\dd,-i}}&\textrm{if } \dd\cdot \mu= 0
\end{cases}
\end{equation}
where $\LLL\coloneqq \HO_c(\mathbb{A}^1,\mathbb{Q})$.  By purity, the virtual Poincar\'e polynomials of the cohomological BPS invariants agree with the above Poincar\'e polynomials, so that we have equalities for the refined BPS invariants
\begin{equation}
\label{refDTa}
\chi_{q^{1/2}}(\DT_{\tilde{Q},\tilde{W}^{\mu},\dd})=\begin{cases}\kac_{Q,\dd}(q^{-1}) & \textrm{if } \dd\cdot \mu\neq 0\\
q^{-1/2}\kac_{Q,\dd}(q^{-1}) &\textrm{if }\dd\cdot \mu=0.
\end{cases}
\end{equation}
\end{thmx}
Comparing with \eqref{KactoDT} we see that for dimension vectors $\dd$ satisfying $\dd\cdot \mu\neq 0$ the cohomological BPS invariants have switched parity.  In particular, for generic choices of $\mu$, the algebra $\Coha_{\tilde{Q},\tilde{W}^{\mu}}$ is a fermionized version of $\Coha_{\tilde{Q},\tilde{W}}$, so \eqref{cohDTa} and \eqref{refDTa} answer Questions \ref{bq} and \eqref{bq2} respectively in the affirmative.
\smallbreak
Still fixing $\mu\in \mathbb{C}^{Q_0}$ the \textit{deformed preprojective algebra}, introduced by Crawley--Boevey and Holland in \cite{CBH98}, is defined by
\[
\Pi_{Q,\mu}\coloneqq\mathbb{C}\overline{Q}/\langle \sum_{a\in Q_1}[a,a^*]+\sum_{i\in Q_0} \mu_i e_i\rangle.
\]
where $e_i$ is the path of length zero beginning and ending at the vertex $i$.  Via the methods used to prove Theorem \ref{main_thm} we are able to calculate the Borel--Moore homology (along with its mixed Hodge structure) of all stacks of representations of deformed preprojective algebras, simultaneously generalising a result of Crawley--Boevey and Van den Bergh \cite{CBVdB04} from the case of generic $\mu$ and indivisible dimension vector $\dd$, and the result from \cite{preproj} which deals with the case $\mu=0$ and arbitrary $\dd$:
\begin{thmx}
\label{defBM}
For arbitrary $\mu\in R$ and $\dd\in\mathbb{N}^{Q_0}$ there is an isomorphism of $\mathbb{N}^{Q_0}$-graded mixed Hodge structures
\begin{equation}
\label{VB}
\bigoplus_{\dd\in \mathbb{N}^{Q_0}}\HO_c\!\left(\Mst_{\dd}(\Pi_{Q,\mu}),\mathbb{Q}\right)\otimes \LLL^{\chi_Q(\dd,\dd)}\cong \Sym\left(\bigoplus_{\substack{0\neq \dd\in\mathbb{N}^{Q_0}\\ \dd\cdot\mu=0}}\DT_{\tilde{Q},\tilde{W},\dd}^{\hdg,\vee}\otimes\HO_c(\pt/\mathbb{C}^*,\mathbb{Q})_{\vir}\right)
\end{equation}
where 
\[
\DT_{\tilde{Q},\tilde{W},\dd}^{\hdg,\vee}\cong\bigoplus_{i\in\mathbb{Z}}(\LLL^{i+1/2})^{\oplus \kac_{Q,\dd,i}}
\]
and
\[
\HO_c(\pt/\CC^*)_{\vir}=\bigoplus_{i\geq 0}\LLL^{-1/2-i}.
\]
In particular, the compactly supported cohomology of $\Mst(\Pi_{Q,\mu})$, the stack of representations of the deformed preprojective algebra, is pure, of Tate type.
\end{thmx}

See \S \ref{bonus_sec} for the proof of Theorems \ref{main_thm} and \ref{defBM}.
\subsection{Acknowledgments}
The suggestion that partial fermionization of BPS Lie algebras should be modelled by the passage from $\Tot(\mathcal{O}_{\mathbb{P}^1}(-2)\oplus \mathcal{O}_{\mathbb{P}^1})$ to $\Tot(\mathcal{O}_{\mathbb{P}^1}(-1)\oplus\mathcal{O}_{\mathbb{P}^1}(-1))$ is based on conversations with, and unpublished work of Kevin Costello.  I am very grateful to him for being so generous with his time and ideas.  I learnt about connections between the various Jacobi algebras appearing in this paper from Victor Ginzburg, and I am similarly grateful to him.  Thanks also to Hans Franzen, Tudor P\u{a}durariu and Markus Reineke for stimulating conversations about DT theory.  My thanks also go to Tommaso Scognamiglio for questions that led to the addition of \S \ref{bonus_sec} to an earlier draft.
\smallbreak
This paper was written for the British Mathematical Colloquium, held (virtually) in Glasgow in 2021.  Many thanks are due to the organisers for managing to run such an engaging online event.  During the writing of the paper, I was supported by the starter grant ``Categorified Donaldson--Thomas theory'' No. 759967 of the European Research Council, and also supported by a Royal Society university research fellowship.
\section{Cohomological DT theory for quivers with potential}
\subsection{Some algebras from quivers}
\label{JAS}
A quiver is determined by a set of vertices $Q_0$, a set of edges $Q_1$, and two morphisms $s,t\colon Q_1\rightarrow Q_0$ taking an arrow to its source and target respectively.  We always assume that $Q_0$ and $Q_1$ are finite.  We define the Euler form
\begin{align}\label{chi_def}
\chi_Q(\cdot,\cdot)\colon &\mathbb{Z}^{Q_0}\times\mathbb{Z}^{Q_0}\rightarrow\mathbb{Z}\\
\nonumber
&(\dd,\ee)\mapsto \sum_{a\in Q_1} \dd_{s(a)}\dd_{t(a)}-\sum_{i\in Q_0}\dd_i\ee_i.
\end{align}
Where there is no possibility of confusion, we drop the quiver $Q$ from the notation and just write $\chi(\cdot,\cdot)$.  We denote by $\overline{Q}$ the doubled quiver of $Q$, obtained by adding an arrow $a^*$ for every arrow $a\in Q_1$, where $a^*$ has the opposite orientation to $a$.  We denote by $\tilde{Q}$ the tripled quiver, obtained from $\overline{Q}$ by adding a loop $\omega_i$ at each vertex $i\in Q_0$.

\smallbreak
Given a ring $A$ and a quiver $Q$ we denote by $A Q$ the free path algebra of $Q$ with coefficients in $A$.  We denote by $R\subset \mathbb{C}Q$ the semisimple subalgebra spanned by length zero paths, so we may identify $R= \CC^{Q_0}$.  We denote by $\Pi_Q$ the preprojective algebra for $Q$, defined to be the quotient of the free path algebra $\mathbb{C}\overline{Q}$ by the two-sided ideal generated by the element $\sum_{a\in Q_1}[a,a^*]$.  As in the introduction, we denote by $\Pi_Q[\omega]$ the trivial extension obtained by adjoining a central element $\omega$ to the algebra $\Pi_Q$, i.e. $\Pi_Q[\omega]=\Pi_Q\otimes \CC[\omega]$.  
\smallbreak
Let $\mu\in R$.  We recall the central extension of $\Pi_Q$ introduced by Etingof and Rains \cite{ER05}:
\[
\Pi_Q^{\mu}=\mathbb{C}[\omega]\overline{Q}/\langle \sum_{a\in Q_1}[a,a^*]+\mu \omega\rangle.
\]
There is an obvious isomorphism 
\[
\Pi_Q^0\cong\Pi_Q[\omega]
\]
and natural isomorphisms
\begin{align*}
\Pi_{Q,\mu}\cong &\Pi_Q^{\mu}/\langle \omega-1\rangle\\
\Pi_{Q}\cong &\Pi_Q^{\mu}/\langle \omega\rangle,
\end{align*}
where $\Pi_{Q,\mu}$ is the deformed preprojective algebra recalled in the introduction.  The algebra $\Pi_Q^{\mu}$ provides an $\mathbb{A}^1$-family of algebras interpolating between the preprojective algebra $\Pi_Q$ and the deformed preprojective algebra $\Pi_{Q,\mu}$.
\smallbreak
Let $Q$ be a quiver and let $W\in\mathbb{C}Q/[\mathbb{C} Q,\mathbb{C} Q]$ be a potential, i.e. a linear combination of cyclic words, where cyclic words are considered to be equivalent if they can be cyclically permuted to each other.  We will call the data of a quiver with potential $(Q,W)$ a QP.  Given $a\in Q_1$, if $W=a_1\ldots a_n$ is a single cyclic word we define
\[
\partial W/\partial a=\sum_{a_m=a}a_{m+1}a_{m+2}\ldots a_na_1\ldots a_{m-1}
\]
and define $\partial W/\partial a$ for general $W$ by extending linearly.  We define
\[
\Jac(Q,W)=\mathbb{C} Q/\langle \partial W/\partial a \;\lvert\; a\in Q_1\rangle.
\]
In this paper we will study Jacobi algebras obtained from the tripled QP $(\tilde{Q},\tilde{W})$ defined in \eqref{wt_def} by adding polynomials in the extra loops $\omega_i$; we refer the reader to \cite[Sec.4]{ginz} for general background on this construction, \cite{Sz08,Vel10} for the noncommutative geometry background in type ADE, and \cite{CKV} for the physics perspective.
\begin{proposition}\cite[Ex.4.3.5]{ginz}
\label{gKW}
Let $\mu=\sum_i \mu_ie_i$.  Set $\tilde{W}^{\mu}=\tilde{W}+\frac{1}{2}\sum_{i\in Q_0} \mu_i\omega_i^2$.  Then there is an isomorphism
\[
\Pi_Q^{\mu}\cong \Jac(\tilde{Q},\tilde{W}^{\mu}).
\]
In particular there is a natural isomorphism
\[
\Pi_Q[\omega]\cong \Jac(\tilde{Q},\tilde{W}).
\]
\end{proposition}
\begin{proof}
This follows more or less from the definitions.  The noncommutative derivatives of $\tilde{W}^{\mu}$ with respect to the arrows $a$ impose the relation that $\omega$ commutes with the arrows $a^*$, and vice versa, while the noncommutative derivatives with respect to the loops $\omega_i$ impose the defining relations of $\Pi^{\mu}_Q$ as a quotient of $\mathbb{C}[\omega]\overline{Q}$.
\end{proof}
\begin{example}
\label{ConEx}
Let $Q$ be defined as in \eqref{Qdef}, and set $\mu=e_0-e_1$.  Then
\[
\tilde{W}^{\mu}=\omega_0(a^*a-bb^*)-\omega_1(aa^*-b^*b)+\frac{1}{2}(\omega_0^2-\omega_1^2).
\]
After the noncommutative change of variables
\begin{align*}
\omega_0&\mapsto \omega_0-a^*a+bb^*\\
\omega_1&\mapsto \omega_1-aa^*+b^*b
\end{align*}
the potential transforms to
\[
W=\frac{1}{2}(\omega_0^2-\omega_1^2)+b^*baa^*-bb^*a^*a
\]
Now the relations $\partial W/\partial \omega_i=\pm\omega_i$ mean that in the Jacobi algebra we may simply remove the loops $\omega_i$.  Thus there is a natural isomorphism
\[
\Jac(\tilde{Q},W)\cong\Jac(Q_{\con},W_{\KW}),
\]
giving the isomorphism \eqref{twoJac}.  In particular, the noncommutative conifold is isomorphic to one of the central extensions of $\Pi_Q$ considered above, for $Q$ as in \eqref{Qdef}.
\end{example}
Later we will use the following elementary result.
\begin{proposition}
\label{sspp}
Let $\rho$ be a $\dd$-dimensional simple $\Jac(\tilde{Q},\tilde{W}^{\mu})$-module.  Then the operator $\rho(\omega)$ acts on the underlying vector space of $\rho$ by multiplication by some scalar $\lambda\in \CC$, and if $\mu\cdot\dd\neq 0$ then $\lambda=0$.
\end{proposition}
\begin{proof}
The first part follows from the fact that $\omega=\sum_{i\in Q_0}\omega_i$ is central in $\Jac(\tilde{Q},\tilde{W}^{\mu})$, so each eigenspace of $\rho(\omega)$ is preserved by the action of $\Jac(\tilde{Q},\tilde{W}^{\mu})$.  For the second part, consider the relation in $\Jac(\tilde{Q},\tilde{W}^{\mu})$
\[
0=\sum_{i\in Q_0}\partial \tilde{W}^{\mu}/\partial \omega_i=\sum_{i\in Q_0}\mu_i \omega_i+ \sum_{a\in Q_1}[a,a^*].
\]
Applying $\rho$ and taking the trace, the final sum vanishes, and we find
\begin{align*}
0=&\sum_{i\in Q_0} \mu_i \lambda\Tr(\Id_{\dd_i\times\dd_i})\\
=&\lambda \mu\cdot \dd
\end{align*}
as required.
\end{proof}
\subsection{Moduli spaces of quiver representations}
Given a quiver $Q$ and a dimension vector $\dd\in\dvs$ we set
\begin{align*}
\mathbb{A}_{\dd}(Q)\coloneqq &\prod_{a\in Q_1}\Hom(\CC^{\dd_{s(a)}},\CC^{\dd_{t(a)}})\\
\Gl_{\dd}\coloneqq & \prod_{i\in Q_0}\Gl_{\dd_i}(\CC).
\end{align*}
The group $\Gl_{\dd}$ acts on $\mathbb{A}_{\dd}(Q)$ via change of basis.  We denote by $\Mst_{\dd}(Q)$ the stack of $\dd$-dimensional $\CC Q$-modules.  There is an isomorphism of stacks
\[
\Mst_{\dd}(Q)\cong \mathbb{A}_{\dd}(Q)/\Gl_{\dd}.
\]
We denote by $\Msp_{\dd}(Q)$ the coarse moduli space of $\dd$-dimensional $\CC Q$-modules.  Geometric $K$-points of $\Msp_{\dd}(Q)$ are in natural bijection with semisimple $KQ$-modules.  There is an isomorphism
\[
\Msp_{\dd}(Q)\cong \Spec(\Gamma(\mathbb{A}_{\dd}(Q))^{\Gl_{\dd}}).
\]
We denote by $\JH_{\dd}\colon \Mst_{\dd}(Q)\rightarrow \Msp_{\dd}(Q)$ the affinization morphism.  Although this morphism is not projective, it is \textit{approximated by projective maps} in the sense of \cite{QEAs}, meaning that $\JH_*$ and $\JH_!$ commute with vanishing cycle functors (introduced in the next section).
\smallbreak
For spaces and morphisms involving a subscript $\dd$, if we omit the subscript, the union over all dimension vectors is intended.
\smallbreak
There is a finite morphism \cite{Meinhardt14}
\[
\oplus\colon \Msp(Q)\times\Msp(Q)\rightarrow \Msp(Q)
\]
which at the level of geometric points, takes a pair of $KQ$-modules to their direct sum.  Since this morphism is invariant under swapping the two factors of $\Msp(Q)$ in the domain, and finite morphisms are exact with respect to the perverse t structure, we obtain an induced symmetric monoidal product on $\Perv(\Msp(Q))$, defined by
\[
\mathcal{F}'\boxtimes_{\oplus}\mathcal{F}''\coloneqq \oplus_*\!\left(\mathcal{F}'\boxtimes\mathcal{F}''\right).
\]
\smallbreak
Given a potential $W\in \mathbb{C}Q/[\mathbb{C} Q,\mathbb{C} Q]$ we form the function $\Tr(W)_{\dd}$ on $\mathbb{A}_{\dd}(Q)$.  This is well defined and $\Gl_{\dd}$-invariant by cyclic invariance of trace.  As such, $\Tr(W)_{\dd}$ induces functions on $\Mst_{\dd}(Q)$ and $\Msp_{\dd}(Q)$, which we continue to denote by $\Tr(W)_{\dd}$, or just $\Tr(W)$ if there is no risk of ambiguity.

\subsection{Cohomological Donaldson--Thomas theory of quivers with potential}

Given a function $f$ on a smooth complex variety $X$ we define $X_0=f^{-1}(0)$, and consider the diagram
\[
\xymatrix{
& \tilde{X}\ar[d]^p\ar[r]& \mathbb{A}^1\ar[d]^{\exp}\\
X_0\ar@{^{(}->}[r]^{\kappa}& X\ar[r] &\mathbb{A}^1.
}
\]
in which the square is Cartesian.  Then we define the nearby cycles functor $\psi_f\colon \Db(\Perv(X))\rightarrow \Db(\Perv(X))$ via
\[
\psi_f=\kappa_*\kappa^*p_*p^*.
\]
The vanishing cycles functor $\phi_f$ is defined so that for $\mathcal{F}\in\Ob(\Db(\Perv(X)))$ there is a distinguished triangle 
\[
\kappa_*\kappa^*\mathcal{F}\rightarrow \psi_f\mathcal{F}\rightarrow \phi_f\mathcal{F}.
\]
Both $\psin{f}\coloneqq \psi_f[-1]$ and $\phin{f}\coloneqq\phi_f[-1]$ send perverse sheaves to perverse sheaves \cite[Cor.10.3.13]{KSsheaves} and (naturally) commute with Verdier duality \cite{Ma09}.  
\smallbreak
We give a lightning account of the critical cohomological Hall algebra associated to a quiver with potential.  More details can be found in \cite{KS2,QEAs}.  For a stack $\Mst$ for which each connected component is irreducible and generically smooth, we define the intersection complex
\[
\ICS_{\Mst}\coloneqq\coprod_{\mathfrak{N}\in\pi_0(\Mst)}\ICS_{\mathfrak{N}}(\mathbb{Q}_{\mathfrak{N}^{\mathrm{sm}}}[\dim(\mathfrak{N})]),
\]
i.e. it is the intermediate extension of the constant perverse sheaf from the smooth locus.
We define 
\[
\rCoha_{Q,W}\coloneqq \JH_*\phin{\Tr(W)}\ICS_{\mathfrak{M}(Q)}.
\]
The morphism $\pi_2$ from \eqref{corr_diag} is proper, so that there is a natural integration map 
\[
\alpha_{\dd',\dd''}\colon\pi_{2,*}\mathbb{Q}_{\mathfrak{E}\mathrm{xact}_{\dd',\dd''}(Q)}\rightarrow \mathbb{Q}_{\mathfrak{M}_{\dd'+\dd''}(Q)}[-2\chi(\dd',\dd'')]
\]
(the shift is given by the relative dimension of $\pi_2$).  Composing appropriate shifts of the morphisms
\begin{align*}
&\oplus_*\JH_*\phin{\Tr(W)}\left(\mathbb{Q}_{\mathfrak{M}(Q)\times\mathfrak{M}(Q)}\rightarrow (\pi_1\times \pi_3)_*\mathbb{Q}_{\mathfrak{E}\mathrm{xact}} \right)
\end{align*}
and the sum of $\JH_*\phin{\Tr(W)}\alpha_{\dd',\dd''}$ over pairs $(\dd',\dd'')$, and using commutativity of vanishing cycle functors with proper and with smooth morphisms, we obtain the morphism
\[
\beta\colon \oplus_*(\JH\times\JH)_*\phin{\Tr(W)}\ICS_{\Mst_{\dd'}(Q)\times\Mst_{\dd''}(Q)}\rightarrow \JH_*\phin{\Tr(W)}\ICS_{\Mst_{\dd}(Q)}.
\]
Finally, composing $\beta$ with $\oplus_*(\JH\times\JH)_*\TS$, where $\TS$ is (a shift of) the Thom--Sebastiani isomorphism \cite{Ma01}
\[
\TS\colon \phin{\Tr(W)}\mathbb{Q}_{\Mst_{\dd'}(Q)}\boxtimes\phin{\Tr(W)}\mathbb{Q}_{\Mst_{\dd''}(Q)}\xrightarrow{\cong} \phin{\Tr(W)}\mathbb{Q}_{\Mst_{\dd'}(Q)\times\Mst_{\dd''}(Q)}
\]
we define the (relative) Hall algebra multiplication
\begin{equation}
\label{relProd}
\rCoha_{Q,W}\boxtimes_{\oplus}\rCoha_{Q,W}\rightarrow \rCoha_{Q,W}.
\end{equation}
The cohomology
\begin{align*}
\Coha_{Q,W}\coloneqq &\HO\!\left(\Mst(Q),\phin{\Tr(W)}\ICS_{\Mst(Q)}\right)\\
\cong&\HO(\Msp(Q),\rCoha_{Q,W})
\end{align*}
has a $\dvs$-grading by dimension vectors induced by the decomposition $\Msp(Q)=\coprod_{\dd\in\dvs}\Msp_{\dd}(Q)$, and the associative product induced by taking derived global sections of the morphism \eqref{relProd} respects this grading.
\smallbreak
Assume that $Q$ is symmetric\footnote{There is a version of the integrality theorem for non-symmetric quivers, concerning vanishing cycle cohomology of stacks of semistable $\mathbb{C}Q$-modules, but we won't need it in this paper.  See \cite[Thm.A Thm.C]{QEAs}}.  For $\dd\in \dvs$ we define 
\begin{equation}
\label{BPSdef}
\DTS_{Q,W,\dd}\coloneqq \begin{cases}\phin{\Tr(W)}\ICS_{\Msp_{\dd}(Q)}& \textrm{if there is a $\dd$-dimensional simple $\mathbb{C}Q$-module}\\ 0&\textrm{otherwise.}\end{cases}
\end{equation}
According to our conventions, $\ICS_{\Msp_{\dd}(Q)}$ is the intermediate extension of the constant perverse sheaf $\mathbb{Q}_{\Msp^{\simp}_{\dd}(Q)}[1-\chi(\dd,\dd)]$ on the (open, dense) subscheme of $\Msp_{\dd}(Q)$ corresponding to simple modules.  Since $\ICS_{\Msp_{\dd}(Q)}$ is Verdier self-dual, and vanishing cycle functors commute with Verdier duality \cite{Ma09}, there are natural isomorphisms
\begin{equation}
\label{nvsd}
\mathbb{D}\DTS_{Q,W,\dd}\cong\DTS_{Q,W,\dd}.
\end{equation}
We recall the following version of the \textit{cohomological integrality theorem} from \cite{QEAs}
\begin{theorem}
\label{IT1}
There is an isomorphism of bounded above complexes of perverse sheaves
\[
\JH_!\phin{\Tr(W)}\ICS_{\Mst(Q)}\cong\Sym_{\boxtimes_{\oplus}}\left(\bigoplus_{\dvs\ni\dd\neq 0}\DTS_{Q,W,\dd}\otimes\HO_c(\pt/\mathbb{C}^*)_{\vir}\right)
\]
where 
\[
\HO_c(\pt/\mathbb{C}^*)_{\vir}\cong \bigoplus_{i\in\mathbb{Z}_{\geq 0}}\mathbb{Q}[1+2i].
\]
\end{theorem}
We define the \textit{cohomological BPS invariants}
\[
\DT_{Q,W,\dd}\coloneqq \HO(\Msp_{\dd}(Q),\DTS_{Q,W,\dd}).
\]
Applying the compactly supported cohomology functor to Theorem \ref{IT1}, and using self-Verdier duality \eqref{nvsd} of $\DTS_{Q,W,\dd}$, yields
\[
\bigoplus_{\dd\in\mathbb{N}^{Q_0}}\HO_c(\Mst_{\dd}(Q),\phin{\Tr(W)}\ICS_{\Mst_{\dd}(Q)})\cong \Sym\left(\bigoplus_{\dvs\ni\dd\neq 0}\DT^{\vee}_{Q,W,\dd}\otimes\HO_c(\pt/\mathbb{C}^*)_{\vir}\right).
\]
The \textit{BPS invariants} of the Jacobi algebra $\Jac(Q,W)$ are defined via
\begin{align*}
\omega_{Q,W,\dd}=&\chi(\HO(\Msp_{\dd}(Q),\DTS_{Q,W,\dd}))\\
\coloneqq &\sum_{i\in\mathbb{Z}}(-1)^i\dim\!\left( \HO^i\!\left(\Msp_{\dd}(Q),\DTS_{Q,W,\dd}\right)\right)\\
=&\chi(\HO_c(\Msp_{\dd}(Q),\DTS_{Q,W,\dd}))
\end{align*}
where the final identity again follows from Verdier self-duality of the BPS sheaf.  Turning to Verdier duals, we have instead
\begin{theorem}\cite{QEAs}
\label{CIT}
There is an isomorphism of unbounded complexes of perverse sheaves
\[
\rCoha_{Q,W}\cong\Sym_{\boxtimes_{\oplus}}\left(\bigoplus_{\dvs\ni\dd\neq 0}\DTS_{Q,W,\dd}\otimes\HO(\pt/\mathbb{C}^*)_{\vir}\right)
\]
where
\[
\HO(\pt/\mathbb{C}^*)_{\vir}=\bigoplus_{i\in\mathbb{Z}_{\geq 0}}\mathbb{Q}[-1-2i],
\]
so that
\[
{}^{\mathfrak{p}}\!\Ho^1(\rCoha_{Q,W})=({}^{\mathfrak{p}}\tau^{\leq 1}\rCoha_{Q,W})[1]\cong \bigoplus_{\dvs\ni\dd\neq 0}\DTS_{Q,W,\dd}.
\]
Applying the natural transformation ${}^{\mathfrak{p}}\tau^{\leq 1}\rightarrow \id$ to $\rCoha_{Q,W}$ and taking hypercohomology, there is a natural inclusion
\[
\mathfrak{g}_{Q,W}\coloneqq \DT_{Q,W}[-1]= \HO(\Msp(Q),\DTS_{Q,W})[-1]\hookrightarrow \Coha_{Q,W}.
\]
The image of this inclusion is closed under the commutator\footnote{Strictly speaking, for this part of the theorem to be true, we need to twist the symmetric monoidal structure on $\dvs$-graded, cohomologically graded vector spaces by a sign, over and above the Koszul sign rule (see \cite[Sec.1.6, Sec.6.1]{QEAs}).  Thankfully for the quiver $\tilde{Q}$ this sign is always $+$ (see \cite[Rem[2.3]{preproj3}).} Lie bracket induced by the associative algebra structure on $\Coha_{Q,W}$.
\end{theorem}
\label{IT2}
The resulting Lie algebra $\mathfrak{g}_{Q,W}$ is called the \textit{BPS Lie algebra} for the pair $(Q,W)$, see \cite{QEAs} for more details.
\subsection{Hodge theoretic BPS invariants}

We give another lightning introduction, this time to Hodge theoretic DT theory, via monodromic mixed Hodge modules.  For more details, we refer the reader to \cite{KS2,QEAs}, and for a comparison with the treatment of monodromic mixed Hodge modules in \cite{Saito10}, we refer the reader to \cite[Sec.2]{QEAs}.  Mixed Hodge structures are important in the subject of refined DT theory, since the extra $q$-variable appearing in refined DT theory keeps track of the weight filtration on certain mixed Hodge structures.  On the other hand, our main result states that all mixed Hodge structures appearing in this paper are pure, so that weight polynomials can be replaced by Poincar\'e polynomials.  The takeaway is that this section can be skimmed by the reader that is happy to use the purity part of Theorem \ref{main_thm} to identify the refined BPS invariants of $\Jac(\tilde{Q},\tilde{W}^{\mu})$ with the Poincar\'e polynomials of the BPS cohomology.
\smallbreak
For $X$ a variety we denote by $\MHM(X)$ the category of mixed Hodge modules on $X$.  There is an equivalence of categories between $\MHM(\pt)$ and the category of graded-polarizable mixed Hodge structures.  Let $\mathcal{B}_X$ denote the full subcategory of $\MHM(X\times\mathbb{A}^1)$ containing those mixed Hodge modules $\mathcal{F}$ such that for each $x\in X$ and $i\in\mathbb{Z}$ the mixed Hodge modules 
\begin{equation}
\label{pbmhm}
\Ho^i\!\left((\{x\}\times\mathbb{A}^1\rightarrow X\times\mathbb{A}^1)^*\mathcal{F}\right)
\end{equation}
are locally constant away from $x\times \{0\}$.  We denote by $\mathcal{C}_X$ the full subcategory of $\mathcal{B}_X$ containing those $\mathcal{F}$ such that each \eqref{pbmhm} is constant.  Equivalently, we may define $\mathcal{C}_X$ as the essential image of $\pi_X^*[1]$, for $\pi_X\colon X\times\mathbb{A}^1\rightarrow X$ the projection.  We denote by $\MMHM(X)$ the Serre quotient $\mathcal{B}_X/\mathcal{C}_X$.  There is an embedding of categories $\MHM(X)\hookrightarrow \MMHM(X)$ defined via 
\[
(X\times\{0\}\hookrightarrow X\times\mathbb{A}^1)_*\colon \MHM(X)\rightarrow \MHM(X\times\mathbb{A}^1).
\]
The direct image functor
\[
\Theta\coloneqq (X\times\CC^*\hookrightarrow X\times\mathbb{A}^1)_*
\]
induces an equivalence of categories between the category of mixed Hodge modules on $X\times \CC^*$ with locally constant cohomology sheaves after restriction to each $\{x\}\times \mathbb{C}^*$ and the category of monodromic mixed Hodge modules on $X$.  Denoting by $\Theta^{-1}$ an inverse equivalence, there is a faithful forgetful functor
\[
\rat\circ (X\xrightarrow{x\mapsto (x,1)}X\times\CC^*)^*\circ\Theta^{-1}[-1]
\]
taking a monodromic mixed Hodge module to its underlying perverse sheaf on $X$.  We abuse notation by denoting this functor also by $\rat$.  For $X$ a variety, we denote by $\underline{\mathbb{Q}}_X$ the lift of the constant sheaf $\mathbb{Q}_X$ to a complex of mixed Hodge modules on $X$.  For $f$ a regular function on $X$ we define the vanishing cycles functor
\begin{align*}
\phim{f}\colon&\MHM(X)\rightarrow \MMHM(X)\\
&\mathcal{F}\mapsto \Theta \underline{\phi}_{u\cdot f}(\mathcal{F}\boxtimes\underline{\mathbb{Q}}_{\CC^*})[1]
\end{align*}
where $u$ is the coordinate on $\CC^*$ and $\underline{\phi}_{u\cdot f}$ is the lift of $\phin{u\cdot f}$ to the categories of mixed Hodge modules.  There is a natural isomorphism $\rat \phim{f}\cong \phin{f}\rat$.
\smallbreak
An object $\mathcal{F}\in\Ob(\MMHM(X))$ inherits a weight filtration from the weight filtration on objects of $\MHM(X\times\mathbb{A}^1)$.  We say that $\mathcal{F}$ is pure of weight $n$ if the associated graded object with respect to this filtration is concentrated in degree $n$.  We say that an object $\mathcal{F}\in\Ob(\Db(\MMHM(X)))$ is \textit{pure} if each $\Ho^i(\mathcal{F})$ is pure of weight $i$.
\smallbreak
The cohomologically graded mixed Hodge structure $\LLL=\HO_c(\mathbb{A}^1,\mathbb{Q})$ is pure: it is concentrated in cohomological degree two, and is pure of weight two.  This object has a tensor square root in $\Db(\MMHM(\pt))$ provided by
\[
\LLL^{1/2}\coloneqq \cone(\underline{\mathbb{Q}}_{\mathbb{A}_1}\rightarrow d_*\underline{\mathbb{Q}}_{\mathbb{A}_1})
\]
where $d\colon \mathbb{A}^1\rightarrow \mathbb{A}^1$ is the morphism $z\mapsto z^2$.  We say a monodromic mixed Hodge structure is of Tate type if it is a direct sum of (possibly negative) tensor powers of the monodromic mixed Hodge structure $\LLL^{1/2}[1]$.
\smallbreak
For $X$ an irreducible variety we denote by
\[
\ICS^{\hdg}_{X}\coloneqq \ICS_X\left(\underline{\mathbb{Q}}_{X^{\mathrm{sm}}}[\dim X]\right)\otimes \LLL^{-\dim(X)/2}[-\dim X]
\]
the natural lift of $\ICS_X$ to a pure weight zero monodromic mixed Hodge module.
\smallbreak
We define 
\begin{equation}
\label{hBPSdef}
\DTS^{\hdg}_{Q,W,\dd}\coloneqq \begin{cases}\phim{\Tr(W)}\ICS^{\hdg}_{\Msp_{\dd}(Q)}& \textrm{if there is a $\dd$-dimensional simple $\mathbb{C}Q$-module}\\ 0&\textrm{otherwise.}\end{cases}
\end{equation}
This is the natural lift of the BPS sheaf to a monodromic mixed Hodge module, i.e. $\rat(\DTS^{\hdg}_{Q,W,\dd})\cong\DTS_{Q,W,\dd}$.  Similarly we define the monodromic mixed Hodge structure
\[
\DT^{\hdg}_{Q,W,\dd}\coloneqq \HO(\Msp_{\dd}(Q),\DTS^{\hdg}_{Q,W,\dd})
\]
satisfying $\rat(\DT^{\hdg}_{Q,W,\dd})\cong\DT_{Q,W,\dd}$.  We define\footnote{Since there is not a fully developed theory of mixed Hodge modules for stacks, some care has to be taken care with this definition.  See \cite{QEAs} for the details.} 
\[
\rCoha_{Q,W}^{\hdg}=\JH_*\phim{\Tr(W)}\ICS_{\Mst(Q)}^{\hdg}.
\]
Since all of the natural transformations defining the multiplication on the Hall algebra $\Coha_{Q,W}$ lift to categories of monodromic mixed Hodge modules \cite{KS2,Da13, QEAs}, as does the Thom--Sebastiani theorem \cite{Saito10}, we may define a multiplication on $\rCoha_{Q,W}^{\hdg}$ that recovers the multiplication on $\rCoha_{Q,W}$ after applying the functor $\rat$.  Likewise, taking derived direct image to a point we obtain the algebra object $\Coha^{\hdg}_{Q,W}$ in monodromic mixed Hodge structures.  Then by \cite{QEAs}, Theorems \ref{IT2} and \ref{IT1} lift to the categories of monodromic mixed Hodge modules and monodromic mixed Hodge structures.  In particular, the BPS Lie algebra $\mathfrak{g}_{Q,W}=\DT_{Q,W}[-1]$ lifts to a Lie algebra object 
\[
\mathfrak{g}^{\hdg}_{Q,W}\coloneqq \DT^{\hdg}_{Q,W}\otimes\LLL^{1/2}
\]
inside $\Db(\MMHM(\pt))$.  This is a Lie subalgebra of $\Coha_{Q,W}^{\hdg}$, considered as a Lie algebra in the category of $\mathbb{N}^{Q_0}$-graded, cohomologically graded monodromic mixed Hodge structures, via the commutator Lie bracket.  See \cite{Da13, QEAs} for full details.
\subsection{Cohomological Donaldson--Thomas theory for preprojective algebras}
In this section we restrict our attention to ``tripled'' QPs of the form $(\tilde{Q},\tilde{W})$, as in the introduction.  Firstly, we recall the following purity result on the BPS cohomology of the Jacobi algebra $\Jac(\tilde{Q},\tilde{W})$:
\begin{theorem}\cite{preproj}
\label{purity_thm}
For an arbitrary quiver $Q$ and dimension vector $\dd\in\dvs$, the mixed Hodge structure $\DT^{\hdg}_{\tilde{Q},\tilde{W},\dd}$ is pure, of Tate type.  In addition (or as a consequence of the cohomological integrality theorem) the mixed Hodge structure on 
\[
\HO_c(\Mst_{\dd}(\tilde{Q}),\phim{\Tr(\tilde{W})}\ICS^{\hdg}_{\Mst_{\dd}(\tilde{Q})})
\]
is pure, of Tate type.
\end{theorem}
We denote by
\begin{equation}
\label{edef}
\tilde{e}\colon \Msp(\Pi_Q)\times\mathbb{A}^1\rightarrow \Msp(\tilde{Q})
\end{equation}
the closed embedding that sends a pair $(\rho,t)$ to the $\CC\tilde{Q}$-module $\rho'$ for which the action of the arrows $a,a^*\in\overline{Q}$ are the same as for $\rho$, and the action of each $\rho'(\omega_i)$ is given by multiplication by $t$.
\smallbreak
We will need the following result on the support and equivariance of the BPS sheaf itself:
\begin{lemma}\cite{preproj}
\label{support_theorem}
For a quiver $Q$ and dimension vector $\dd\in\dvs$ there is a perverse sheaf 
\[
\DTS_{\Pi_Q,\dd}\in \Perv(\Msp_{\dd}(\Pi_Q))
\]
such that there is an isomorphism
\[
\DTS_{\tilde{Q},\tilde{W},\dd}\cong \tilde{e}_!\left(\DTS_{\Pi_Q,\dd}\boxtimes \ICS_{\mathbb{A}^1}\right).
\]
The same result holds at the level of monodromic mixed Hodge modules.
\end{lemma}
In words, the theorem says that the BPS sheaf is supported on the subspace of $\CC \tilde{Q}$-modules for which all of the generalised eigenvalues of the operators $\rho(\omega_i)$ are the same, and the sheaf is moreover equivariant for the $\mathbb{A}^1$-action that acts by adding scalar multiples of the identity to all of the operators $\rho(\omega_i)$ simultaneously.
\smallbreak
We recall from \cite{preproj} the description of the BPS cohomology of the Jacobi algebra $\Jac(\tilde{Q},\tilde{W})$ in terms of Kac polynomials:
\begin{theorem}
\label{DTKac}
The Poincar\'e polynomials of the cohomological BPS invariants for the QP $(\tilde{Q},\tilde{W})$ satisfy
\begin{equation}
\label{main_poin}
\poin(\DT_{\tilde{Q},\tilde{W},\dd},q^{1/2})=q^{-1/2}\kac_{Q,\dd}(q^{-1})
\end{equation}
where 
\[
\kac_{Q,\dd}(q)=\sum_{i\in\mathbb{Z}_{\geq 0}} \kac_{Q,\dd,i}\:q^i
\]
is the Kac polynomial, counting the number of isomorphism classes of absolutely\footnote{A module is called \textit{absolutely} indecomposable if it remains indecomposable after extending scalars to the algebraic closure $\overline{\mathbb{F}}_q$.} indecomposable $\mathbb{F}_qQ$-modules for $\mathbb{F}_q$ a field of order $q$.  Furthermore the natural mixed Hodge structure on $\DT_{\tilde{Q},\tilde{W},\dd}$ is pure, of Tate type, so that we can write
\begin{equation}
\label{HDT}
\DT^{\hdg}_{\tilde{Q},\tilde{W},\dd}\cong \bigoplus_{i\in\mathbb{Z}_{\geq 0}}(\LLL^{i-1/2})^{\oplus \kac_{Q,\dd,-i}}
\end{equation}
and so
\begin{equation}
\label{Kacg}
\mathfrak{g}^{\hdg}_{\tilde{Q},\tilde{W},\dd}\cong \bigoplus_{i\in\mathbb{Z}_{\geq 0}}(\LLL^{i})^{\oplus \kac_{Q,\dd,-i}}.
\end{equation}
\end{theorem}

\section{Deformed dimensional reduction and proofs of main results}
\subsection{Deformed dimensional reduction}
The main tool in proving Theorem \ref{main_thm} will be \textit{deformed dimensional reduction}, as introduced in joint work with Tudor P\u{a}durariu \cite{DaPa20}.  This is a geometric result about vanishing cycle functors for functions satisfying certain $\mathbb{C}^*$-equivariance properties.  We state the version that we need below.
\begin{theorem}\cite[Thm.1.3]{DaPa20}
\label{DDR_theorem}
Let the algebraic group $G$ act on a variety $X$ and affine space $\mathbb{A}^n$.  Assume that $\mathbb{A}^n$ is also given a $\mathbb{C}^*$-action, with non-negative weights, which commutes with the $G$-action.  Let $\CC^*$ act on $\overline{X}=X\times\mathbb{A}^n$ via the product of the given action on $\mathbb{A}^n$ with the trivial action on $X$.  Let $g$ be a function on $\overline{X}$ that is $G$-invariant and $\CC^*$-semi-invariant, with strictly positive weight.  Assume that we are given a $G\times\CC^*$-equivariant decomposition $\mathbb{A}^n=\mathbb{A}^m\times \mathbb{A}^{n-m}$, and that we can write
\[
g=g_0+\sum_{1\leq j\leq m} g_jt_j
\]
where the functions $g_0,\ldots g_m$ are pulled back from $X\times \mathbb{A}^{n-m}$ and $t_1,\ldots,t_m$ are a system of coordinates for $\mathbb{A}^m$.  Let $Z\subset X\times \mathbb{A}^{n-m}$ be the vanishing locus of the functions $g_1,\ldots,g_m$.  Then $Z$ is $G$-invariant.  Set $\overline{Z}=Z\times\mathbb{A}^{m}\subset \overline{X}$.  We denote by 
\begin{align*}
\pi\colon &\overline{X}\rightarrow X\\
q\colon&X\times\mathbb{A}^{n-m}\rightarrow X
\end{align*}
the natural projections.  Then the natural transformation
\begin{equation}
\label{CU}
\pi_!\phin{g}\mathbb{Q}_{\overline{X}/G}\rightarrow \pi_!\phin{g_0}\mathbb{Q}_{\overline{Z}/G}\cong q_!\phin{g_0}\mathbb{Q}_{Z/G}[-2m]
\end{equation}
is an isomorphism.
\smallbreak
Since the functor $\rat$ is faithful, the same statement is true at the level of (monodromic) mixed Hodge modules: the natural transformation
\[
\pi_!\phim{g}\underline{\mathbb{Q}}_{\overline{X}/G}\rightarrow \pi_!\phim{g_0}\underline{\mathbb{Q}}_{\overline{Z}/G}\cong q_!\phim{g_0}\underline{\mathbb{Q}}_{Z/G}\otimes \LLL^{m}
\]
is an isomorphism in\footnote{The derived category of monodromic mixed Hodge modules on this global quotient stack is defined following e.g. \cite{emhm}.} $\Db(\MMHM(X/G))$.
\end{theorem}

\subsection{Proof of Theorem \ref{main_thm}}
\label{MPS}
We proceed by applying deformed dimensional reduction to the function $\Tr(\tilde{W}^{\mu})$.  
\begin{proof}[Proof of Theorem \ref{main_thm}]Given a quiver $Q$ we denote by $Q^+$ the quiver obtained by adding a loop $\omega_i$ at each vertex $i\in Q_0$, or equivalently the quiver obtained by removing all of the arrows $a^*$ from $\tilde{Q}$.  We denote by $\mathfrak{I}(Q)$ the stack of pairs $(\rho,f)$, consisting of a $\CC Q$-module $\rho$ and an endomorphism $f\in \Hom_{\CC Q}(\rho,\rho)$.  Then we consider the commutative diagram
\[
\xymatrix{
&\mathfrak{I}(Q)\ar@{^{(}->}[d]^i\\
\Mst(\tilde{Q})\ar[d]^{\tilde{\JH}}\ar[r]^p&\Mst(Q^+)\ar[d]^{\JH^+}\ar[r]^-q&\Mst(Q)\ar[d]^{\JH}\\
\Msp(\tilde{Q})\ar[r]^{p'}&\Msp(Q^+)\ar[r]^-{q'}&\Msp(Q)\ar[dr]^{\tau}\\
\Msp(\Pi_Q)\times\mathbb{A}^1\ar@{_{(}->}[u]_{\tilde{e}}\ar[r]^h&\Msp(Q)\times\mathbb{A}^1\ar@{_{(}->}[u]_e&&\pt.
}
\]
In the above diagram, $i$ is the natural embedding of stacks, sending a pair $(\rho,f)$ to the $\CC Q^+$-module $\rho'$ for which the underlying $\CC Q$-module is $\rho$, and the action of the loops $\omega_i$ is given by $f$.  All of the horizontal arrows are the natural forgetful maps.  We set 
\begin{align*}
\pi=&qp\\
\pi'=&q'p'.
\end{align*}
We denote by $\Omega$ the quiver containing the same vertices as $Q$, and for which the only arrows are the loops $\omega_i$.  We denote by $Q^{\op}$ the quiver containing the same vertices as $Q$, and only the arrows $a^*$ for $a\in Q_1$, i.e. $Q^{\op}$ is the opposite quiver to $Q$.  In the notation of Theorem \ref{DDR_theorem} we set 
\begin{align}
\label{3waysplit}
X=&\mathbb{A}_{\dd}(Q)\\ \nonumber
\mathbb{A}^m=&\mathbb{A}_{\dd}(Q^{\op})\\ \nonumber
\mathbb{A}^{n-m}=&\mathbb{A}_{\dd}(\Omega)\\ \nonumber
G=&\Gl_{\dd}.
\end{align}
We let $\CC^*$ act on $\AA^m$ and $\AA^{n-m}$ with weight one.  
\smallbreak
We set 
\[
L_{\mu}=\frac{1}{2}\sum_{i\in Q_0}\mu_i\omega_i^2
\]
so $\tilde{W}^{\mu}=\tilde{W}+L_{\mu}$.  Then $\Tr(\tilde{W}+L_{\mu})$ satisfies the conditions of Theorem \ref{DDR_theorem}, for example it is $\CC^*$ semi-invariant with weight two.  In the notation of that theorem $Z\subset X\times \mathbb{A}^{n-m}=\mathbb{A}_{\dd}(Q^+)$ is the locus containing those $\CC Q^+$-modules $\rho$ such that the endomorphisms $\rho(\omega_i)$ determine an endomorphism of the underlying $\CC Q$-module of $\rho$.  It follows that $Z/G=\mathfrak{Z}(G)$.  By Theorem \ref{DDR_theorem} there is an isomorphism\footnote{Note that the shift $2\sum_{a\in Q_1}\dd_{s(a)}\dd_{t(a)}$ appearing in the definition of the intersection complex is equal to the $2m$ appearing in \eqref{CU}.}
\begin{align}
\label{DR1} \pi_!\phin{\Tr(\tilde{W}^{\mu})}\ICS_{\Mst(\tilde{Q})}\cong& q_!\phin{\Tr(L_{\mu})}\mathbb{Q}_{\mathfrak{Z}(Q)}.
\end{align}
If instead we set
\begin{align*}
X=&\mathbb{A}_{\dd}(Q^+)\\
\mathbb{A}^m=&\mathbb{A}_{\dd}(\Omega)
\end{align*}
set $n=m$, and again let $\CC^*$ act on $\mathbb{A}^m$ with weight one, then the function $\Tr(\tilde{W})$ still satisfies the conditions of Theorem \ref{DDR_theorem} (now it is semi-invariant with weight one).  So instead we arrive at the isomorphism
\begin{align}
\label{DR2} p_!\phin{\Tr(\tilde{W})}\ICS_{\Mst(\tilde{Q})}\cong& \mathbb{Q}_{\mathfrak{I}(Q)}.
\end{align}

Since $\JH^+$ is approximated by proper maps, there is a natural isomorphism (see \cite[Sec.4.1]{QEAs})
\begin{equation}
\label{ncomm}
\JH^+_!\phin{\Tr(L_{\mu})}\mathbb{Q}_{\mathfrak{I}(Q)}\cong \phin{\Tr(L_{\mu})}\JH^+_!\mathbb{Q}_{\mathfrak{I}(Q)}.
\end{equation}

Combining all of the above we can write\footnote{For the final isomorphism, see \cite[Prop.3.11]{QEAs} for the required compatibility with the symmetrizing morphism.}
\begin{align*}
\pi'_!\tilde{\JH}_!\phin{\Tr(\tilde{W}^{\mu})}\ICS_{\Mst(\tilde{Q})}\cong & \JH_!\pi_!\phin{\Tr(\tilde{W}^{\mu})}\ICS_{\Mst(\tilde{Q})}&
\\\cong&\JH_!q_!\phin{\Tr(L_{\mu})}\mathbb{Q}_{\mathfrak{I}(Q)} &\textrm{by }\eqref{DR1}\\
\cong &\JH_!q_!\phin{\Tr(L_{\mu})}p_!\phin{\Tr(\tilde{W})}\ICS_{\Mst(\tilde{Q})}&\textrm{by }\eqref{DR2}\\
\cong &q'_!\JH^+_!\phin{\Tr(L_{\mu})}p_!\phin{\Tr(\tilde{W})}\ICS_{\Mst(\tilde{Q})}\\
\cong &q'_!\phin{\Tr(L_{\mu})}\JH^+_!p_!\phin{\Tr(\tilde{W})}\ICS_{\Mst(\tilde{Q})}&\textrm{by } \eqref{ncomm}\\
\cong &q'_!\phin{\Tr(L_{\mu})}p'_!\tilde{\JH}_!\phin{\Tr(\tilde{W})}\ICS_{\Mst(\tilde{Q})}\\
\cong &q'_!\phin{\Tr(L_{\mu})}p'_!\Sym_{\oplus}\left(\bigoplus_{\dd\in\dvs\setminus 0} \DTS_{\tilde{Q},\tilde{W},\dd}\otimes\HO_c(\pt/\CC^*)
_{\vir}\right)&\textrm{by Theorem }\ref{IT1}\\
\cong &q'_!\phin{\Tr(L_{\mu})}\Sym_{\oplus}\left(\bigoplus_{\dd\in\dvs\setminus 0} p'_!\DTS_{\tilde{Q},\tilde{W},\dd}\otimes\HO_c(\pt/\CC^*)
_{\vir}\right)&\textrm{since }p'\textrm{ is a monoid map}\\
\cong &q'_!\Sym_{\oplus}\left(\bigoplus_{\dd\in\dvs\setminus 0} \phin{\Tr(L_{\mu})}p'_!\DTS_{\tilde{Q},\tilde{W},\dd}\otimes\HO_c(\pt/\CC^*)
_{\vir}\right)&\textrm{by Thom--Sebastiani.}
\end{align*}
By Lemma \ref{support_theorem} we can write
\begin{equation}
\label{tsev}
p'_!\DTS_{\tilde{Q},\tilde{W},\dd}\cong e_!(\mathcal{F}\boxtimes \ICS_{\mathbb{A}^1})
\end{equation}
where
\[
\mathcal{F}=(\Msp(\Pi_Q)\rightarrow \Msp(Q))_!\DTS_{\Pi_Q,\dd}.
\]
Writing $f=\Tr(L_{\mu})\lvert_{\Msp_{\dd}(Q)\times\mathbb{A}^1}$, we have
\begin{align*}
f\colon &\Msp_{\dd}(Q)\times\mathbb{A}^1\rightarrow\CC\\
&(\rho,t)\mapsto \frac{1}{2}(\mu\cdot \dd)t^2.
\end{align*}
It follows from \eqref{tsev} that
\[
\phin{\Tr(L_{\mu})}p'_!\DTS_{\tilde{Q},\tilde{W},\dd}\cong\begin{cases}e_!(\mathcal{F}\boxtimes \mathbb{Q}_0)&\textrm{if }\mu\cdot \dd\neq 0\\
 e_!(\mathcal{F}\boxtimes \ICS_{\mathbb{A}^1})&\textrm{if }\mu\cdot \dd= 0,
\end{cases}
\]
so that 
\[
\HO_c(\Msp_{\dd}(Q^+),\phin{\Tr(L_{\mu})}p'_!\DTS_{\tilde{Q},\tilde{W},\dd})\cong \HO_c(\Msp_{\dd}(\Pi_Q),\DTS_{\Pi_Q,\dd})[-g(\dd)]
\]
where we define
\[
g(\dd)=\begin{cases} 1&\textrm{if }\mu\cdot\dd=0\\
0&\textrm{if }\mu\cdot \dd\neq 0.
\end{cases}
\]
Finally we deduce that 
\begin{align}
\HO_c(\Mst(\tilde{Q}),\phin{\Tr(\tilde{W}^{\mu})}\ICS_{\Mst(\tilde{Q})})=&\tau_!\pi'_!\tilde{\JH}_!\phin{\Tr(\tilde{W}^{\mu})}\ICS_{\Mst(\tilde{Q})}\nonumber\\
\cong &\Sym\left(\bigoplus_{\dd\in\dvs\setminus 0}\left(\HO_c(\Msp_{\dd}(\Pi_Q),\DTS_{\Pi_Q,\dd})\otimes \HO_c(\pt/\mathbb{C}^*)_{\vir} \right)[-g(\dd)]\right).
\label{firstround}
\end{align}
On the other hand, by Theorem \ref{IT1} again, we have
\begin{equation}
\HO_c(\Mst(\tilde{Q}),\phin{\Tr(\tilde{W}^{\mu})}\ICS_{\Mst(\tilde{Q})})\cong \Sym\left(\bigoplus_{\dd\in\dvs\setminus 0}\DT^{\vee}_{\tilde{Q},\tilde{W}^{\mu},\dd}\otimes \HO_c(\pt/\CC^*)_{\vir}\right).
\label{secondround}
\end{equation}
Comparing \eqref{firstround} and \eqref{secondround} we find
\begin{align*}
\DT_{\tilde{Q},\tilde{W}^{\mu},\dd}\cong &(\HO_c(\Msp_{\dd}(\Pi_Q),\DTS_{\Pi_Q,\dd}))^{\vee}[g(\dd)]\\
\cong &(\HO_c(\Msp_{\dd}(\tilde{Q}),\DTS_{\tilde{Q},\tilde{W},\dd})[1])^{\vee}[g(\dd)]& \textrm{by Theorem \ref{support_theorem}}\\
\cong &\HO_c(\Msp_{\dd}(\tilde{Q}),\DTS_{\tilde{Q},\tilde{W},\dd})^{\vee}[g(\dd)-1]\\
\cong &\DT_{\tilde{Q},\tilde{W},\dd}[g(\dd)-1]
\end{align*}
and now \eqref{main_poin} follows from Theorem \ref{DTKac}.
\smallbreak
Performing the same calculations in the category of monodromic mixed Hodge structures, isomorphisms \eqref{firstround} and \eqref{secondround} yield the isomorphism of $\dvs$-graded complexes of monodromic mixed Hodge structures
\begin{equation}
\label{thirdround}
\Sym\left(\bigoplus_{\dd\in\dvs\setminus 0}\DT^{\hdg,\vee}_{\tilde{Q},\tilde{W},\dd}\otimes \LLL^{(g(\dd)-1)/2}\otimes \HO_c(\pt/\mathbb{C}^*)_{\vir} \right)\cong \Sym\left(\bigoplus_{\dd\in\dvs\setminus 0}\DT^{\hdg,\vee}_{\tilde{Q},\tilde{W}^{\mu},\dd}\otimes \HO_c(\pt/\CC^*)_{\vir}\right)
\end{equation}
where
\[
\HO_c(\pt/\CC^*)_{\vir}=\bigoplus_{i\geq 0}\LLL^{-1/2-i}.
\]
On the other hand, by Theorem \ref{purity_thm} the mixed Hodge structure $\DT^{\hdg}_{\Pi_Q,\dd}$ is pure, so that $\DT^{\hdg,\vee}_{\Pi_Q,\dd}\otimes\LLL^{n/2}$ is pure for all $n\in\mathbb{Z}$.  It follows that both sides of \eqref{thirdround} are pure.  The isomorphism
\[
\DT^{\hdg}_{\tilde{Q},\tilde{W},\dd}\otimes \LLL^{(g(\dd)-1)/2}\cong \DT^{\hdg}_{\tilde{Q},\tilde{W}^{\mu},\dd}
\]
then follows from \eqref{thirdround} and semisimplicity of the category of pure monodromic mixed Hodge structures as in \cite[Cor.7.1]{DaPa20}. \end{proof}
\begin{example}
To recover the example that we started the paper with, consider the quiver $Q^{(0)}$ with one vertex and no arrows.  We label the unique arrow of the one-loop quiver $Q^{(1)}=\widetilde{Q^{(0)}}$ by $\omega_0$.  For $\mu=\mu_0e_0$ to be generic we just have to pick $\mu_0\neq 0$.  Then we find
\begin{align*}
\Jac(Q^{(1)},\tilde{W}^{\mu})=&\Jac(Q^{(1)},\omega_0^2)\\
\cong& \CC Q^{(0)}
\end{align*}
and so the fermionic version of $\Coha_{\widetilde{Q^{(0)}},\tilde{W}}=\Coha_{Q^{(1)}}$ is indeed $\Coha_{Q^{(0)}}$.
\end{example}
\begin{example}
\label{nccf}
We return once more to the noncommutative conifold.  As observed in Example \ref{ConEx}, this algebra is obtained by setting $\mu=(1,-1)$ for the quiver \eqref{Qdef} and considering $\Jac(\tilde{Q},\tilde{W}^{\mu})$.  The cohomological BPS invariants of the QP $(\tilde{Q},\tilde{W})$ are given in \eqref{h0DT}, while the cohomological BPS invariants of the noncommutative conifold are given in \eqref{h1DT}.  Comparing the two, there is a cohomological shift between the respective $(m,n)$th cohomological BPS invariants precisely if $m\neq n$, i.e. precisely if $\mu\cdot(m,n)\neq 0$.
\smallbreak
Recall from the introduction that the BPS invariants for the resolved conifold are only a \textit{partially} fermionized version of the BPS invariants of $Y_0\times\mathbb{A}^1$.  In view of the main result, we see that the generic deformation in the commutative algebraic geometry context corresponds to the deformation $\mu=(1,-1)$.  To fully fermionize the DT theory we are obliged to work in the fully noncommutative\footnote{I.e. with a Jacobi algebra that is not derived equivalent to a threefold.} context provided by generic $\mu$.  Note that deformations within algebraic geometry of the Kleinian singularity were parameterised by $\mathfrak{h}$, the Cartan subalgebra of the (reduced) McKay graph $\Gamma'$, whereas noncommutative deformations are parameterised by the Cartan subalgebra of the full McKay graph.
\end{example}
\begin{remark}
\label{changeN}
Along the course of the proof of Theorem \ref{main_thm} we have shown that
\[
\pi'_*\DTS_{\tilde{Q},\tilde{W}^{\mu},\dd}\cong \begin{cases} \pi'_*\DTS_{\tilde{Q},\tilde{W}}[-1] &\textrm{if }\mu\cdot \dd\neq 0\\
\pi'_*\DTS_{\tilde{Q},\tilde{W}}&\textrm{otherwise.}\end{cases}
\]
With a little effort, one may lift this statement to the level of monodromic mixed Hodge module complexes.
\end{remark}

\subsection{Proof of Theorem \ref{defBM}}
\label{bonus_sec}
Given a quiver $Q$, an element $\mu\in R$, and a number $n\in\mathbb{Z}_{\geq 1}$, define 
\[
\tilde{W}^{\mu}_n=\tilde{W}+\frac{1}{n}\mu\omega^n.
\]
Then in the decomposition \eqref{3waysplit}, we let $\CC^*$ act with weight $1$ on $\AA_{\dd}(\Omega)$, weight $n-1$ on $\AA_{\dd}(Q^{\opp})$, and trivially on $\AA_{\dd}(Q)$, so $\Tr(\tilde{W}^{\mu}_{\dd})$ satisfies the conditions of Theorem \ref{DDR_theorem} (now it is a weight $n$ function).  Then the argument of \S \ref{MPS} gives that
\begin{equation}
\label{general_n}
\BPS^{\hdg}_{\tilde{Q},\tilde{W}^{\mu}_n,\dd}\cong \BPS^{\hdg}_{\tilde{Q},\tilde{W},\dd}\otimes \LLL^{1/2}\otimes \HO(\AA^1,\phim{(\mu\cdot\dd) x^n}\ICS_{\AA^1}).
\end{equation}
Theorem \ref{main_thm} follows from the special case $n=2$ and Theorem \ref{DTKac}, observing that $\HO(\AA^1,\phim{x^2}\underline{\mathbb{Q}}_{\AA^1})=\LLL^{1/2}$.

In this section we consider instead the special case $n=1$, i.e. we consider the quiver $\tilde{Q}$ with the potential
\[
\tilde{W}_1^{\mu}=\sum_{a\in Q_1}[a,a^*]+\mu\omega.
\]
This potential is linear in the loops $\omega_i$, so that the following proposition is a straightforward application of (undeformed) dimensional reduction:
\begin{proposition}
\label{dppdr}
There is an isomorphism in the derived category of complexes of mixed Hodge modules
\[
\pi_!\phim{\Tr(\tilde{W}_1^{\mu})}\underline{\mathbb{Q}}_{\Mst_{\dd}(\tilde{Q})}\rightarrow \underline{\mathbb{Q}}_{\Mst_{\dd}(\Pi_{Q,\mu})}\otimes\LLL^{\dd\cdot\dd}.
\]
\end{proposition}
\begin{proof}
In the setup of Theorem \ref{DDR_theorem}, we put $X=\mathbb{A}_{\dd}(\overline{Q})$, $n=m$, $\mathbb{A}^m=\mathbb{A}_{\dd}(\Omega)$ and $G=\Gl_{\dd}$.  Then $g_0$ is the zero function, so that there is a natural isomorphism $\phim{g_0}\rightarrow \id$, and the result follows from Theorem \ref{DDR_theorem} and $\dim(\mathbb{A}_{\dd}(\Omega))=\dd\cdot\dd$.
\end{proof}
\begin{proof}[Proof of Theorem \ref{defBM}]
We have isomorphisms of $\mathbb{N}^{Q_0}$-graded mixed Hodge structures
\begin{align*}
\bigoplus_{\dd\in \mathbb{N}^{Q_0}}\HO_c\!\left(\Mst_{\dd}(\Pi_{Q,\mu}),\mathbb{Q}\right)\otimes \LLL^{\chi_Q(\dd,\dd)}\cong &\bigoplus_{\dd\in\mathbb{N}^{Q_0}}\HO_c\left(\Mst_{\dd}(\tilde{Q}),\phim{\Tr(W_1^{\mu})}\ICS_{\Mst_{\dd}(\tilde{Q})}\right)\\
\cong&\Sym\left(\bigoplus_{\mathbb{N}^{Q_0}\ni \dd\neq 0}\DT^{\hdg,\vee}_{\tilde{Q},W_1^{\mu},\dd}\otimes\HO_c(\pt/\mathbb{C}^*)_{\vir}\right)
\end{align*}
where the first isomorphism is dimensional reduction (as in Proposition \ref{dppdr}) and the second is the cohomological integrality theorem.  So the theorem follows from \eqref{HDT} and the claim that
\[
\DT^{\hdg}_{\tilde{Q},W_1^{\mu},\dd}\cong\begin{cases} \DT^{\hdg}_{\tilde{Q},\tilde{W},\dd} &\textrm{if }\dd\cdot\mu=0\\
0& \textrm{otherwise.}\end{cases}
\]
This follows from \eqref{general_n} and the observation that for $\lambda\in\mathbb{C}$
\[
\HO\!\left(\mathbb{A}^1,\phim{\lambda x}\ICS_{\mathbb{A}^1}\right)\cong\begin{cases} \LLL^{-1/2}& \textrm{if }\lambda=0\\
0&\textrm{if }\lambda\neq 0.\end{cases}
\]
\end{proof}

Assume that the dimension vector $\dd$ is indivisible, meaning that there is no $n\in \mathbb{Z}_{\geq 2}$ such that $\frac{1}{n}\dd\in\mathbb{N}^{Q_0}$.  Assume also that $\mu$ is chosen to be generic, subject to the constraint that $\dd\cdot\mu=0$.  Equivalently, $\dd$ and $\mu$ are chosen so that $\dd'\cdot \mu=0$ implies that $\dd'$ is an integer multiple of $\dd$.  Then \eqref{VB} simplifies to
\begin{equation}
\label{VBs}
\bigoplus_{n\in\mathbb{Z}_{\geq 0}}\HO_c\!\left(\Mst_{n \dd}(\Pi_{Q,\mu}),\mathbb{Q}\right)\otimes \LLL^{n^2\chi_Q(\dd,\dd)}\cong \Sym\left(\bigoplus_{n\in\mathbb{Z}_{>0}}\DT_{\tilde{Q},\tilde{W},n\dd}^{\hdg,\vee}\otimes\HO_c(\pt/\mathbb{C}^*,\mathbb{Q})_{\vir}\right).
\end{equation}
Our assumptions imply that $\Msp_{\dd}(\Pi_Q)$ is a fine moduli scheme, and $\Mst_{\dd}(\Pi_Q)$ is a trivial $\mathrm{B}\mathbb{C}^*$-gerbe over it, meaning that
\[
\HO_c\!\left(\Mst_{\dd}(\Pi_{Q,\mu}),\mathbb{Q}\right)\cong \HO_c(\Msp_{\dd}(\Pi_Q),\mathbb{Q})\otimes \HO_c(\pt/\mathbb{C}^*,\mathbb{Q}).
\]
Recall that there is an isomorphism
\[
\HO_c(\pt/\mathbb{C}^*,\mathbb{Q})\cong\bigoplus_{i\geq 0}\LLL^{-i-1}.
\]
From the $n=1$ piece of \eqref{VBs} we deduce the following slightly stronger version of a result of Crawley--Boevey and Van den Bergh \cite{CBVdB04}:
\begin{corollary}
Let $\dd$ be indivisible, and $\mu$ be generic.  Then there is an isomorphism of mixed Hodge structures
\[
\HO_c(\Msp_{\dd}(\Pi_{Q,\mu}),\mathbb{Q})\cong \bigoplus_{i\geq 0}(\LLL^{1+i- \chi_Q(\dd,\dd)})^{\oplus \kac_{Q,\dd,i}}.
\]
In particular, $\HO_c(\Msp_{\dd}(\Pi_{Q,\mu}),\mathbb{Q})$ is pure of Tate type.
\end{corollary}

\section{Further directions}
\subsection{Calculating the BPS sheaves}
While our main theorem gives a way to calculate the BPS cohomology $\DT_{\tilde{Q},\tilde{W}^{\mu},\dd}$ for arbitrary $Q$, $\dd\in \dvs$ and $\mu\in R$, the actual BPS sheaf $\DTS_{\tilde{Q},\tilde{W}^{\mu},\dd}$ remains a little mysterious.  We can at least generalise (part of) the support lemma (Lemma \ref{support_theorem}) from the case $\mu=0$.  
\begin{lemma}
For a quiver $Q$, and dimension vector $\dd\in\dvs$, we have
\begin{align*}
\supp(\DTS_{\Pi^{\mu}_Q,\dd})\subset\begin{cases}\tilde{e}(\Msp_{\dd}(\Pi_Q)\times\mathbb{A}^1)& \textrm{if } \mu\cdot \dd=0\\
\tilde{e}(\Msp_{\dd}(\Pi_Q)\times\{0\})&\textrm{if }\mu\cdot\dd\neq 0\end{cases}
\end{align*}
where $\tilde{e}$ is defined in \eqref{edef}.
\end{lemma}
\begin{proof}
The proof is very similar to \cite[Lem.4.1]{preproj}.  Let $\rho$ be a $\Jac(\tilde{Q},\tilde{W}^{\mu})$-module lying in the support of $\DTS_{\tilde{Q},\tilde{W}^{\mu},\dd}$.  By definition, $\rho$ is semisimple, and so since $\rho(\sum_{i\in Q_0}\omega_i)$ is central, it acts via a diagonal matrix.  Arguing as in \cite[Lem.4.1]{preproj} and using centrality of $\rho(\sum_{i\in Q_0}\omega_i)$ along with the cohomological integrality theorem, we deduce that the generalised eigenvalues of $\rho(\omega)$ are all the same, i.e. $\rho(\omega)$ acts via scalar multiplication, establishing the lemma in the case $\mu\cdot\dd\neq 0$.  Arguing as in Proposition \ref{sspp}, if $\mu\cdot\dd\neq 0$, we must have $\lambda=0$.  This establishes the lemma in the case $\mu\cdot\dd\neq 0$.
\end{proof}
Even in the case $\mu\cdot \dd=0$ this is a slightly weaker statement than Lemma \ref{support_theorem}; it is harder to show that $\DTS_{\tilde{Q},\tilde{W}^{\mu},\dd}$ is $\mathbb{A}^1$-equivariant (where $\mathbb{A}^1$ acts by adding a scalar multiple of the identity matrix to $\rho(\omega)$) since the function $\Tr(\tilde{W}^{\mu})$ is not $\mathbb{A}^1$-invariant unless $\mu=0$.
\smallbreak
It would be interesting to compare $\DTS_{\tilde{Q},\tilde{W}^{\mu},\dd}$ with $\DTS_{\tilde{Q},\tilde{W},\dd}\cong \tilde{e}_!\left(\DTS_{\Pi_Q,\dd}\boxtimes \ICS_{\mathbb{A}^1}\right)$ in the case $\mu\cdot\dd=0$, and with $\tilde{e}_!\left(\DTS_{\Pi_Q,\dd}\boxtimes \mathbb{Q}_0\right)$ in the case $\mu\cdot \dd\neq 0$.  In particular, we may ask the following question:
\begin{question}
\label{PQ}
Is the monodromic mixed Hodge module $\DTS_{\tilde{Q},\tilde{W}^{\mu},\dd}$ pure?
\end{question}
By \cite[Thm.A]{preproj3}, we know that $\DTS_{\tilde{Q},\tilde{W},\dd}$ is a pure monodromic mixed Hodge module, i.e. we know that for $\mu=0$ the answer to Question \ref{PQ} is yes.
\subsection{The BPS algebra}
Identifying the BPS cohomology of the Jacobi algebra $\CC(\tilde{Q},\tilde{W}^{\mu})$ is only part of understanding its cohomological DT theory; it tells us the size of the graded pieces of $\Coha_{\tilde{Q},\tilde{W}^{\mu}}$, but nothing about the algebra structure above what we already know from \cite{QEAs} regarding general quivers with potential (e.g. the PBW theorem).  At least for $\mu=0$, the algebra $\Coha_{\tilde{Q},\tilde{W}^{\mu}}$ has been studied from various points of view (see e.g. \cite{Nak01,Var00,MO19,KV19, DPS20}) while for the $\mu\neq 0$ case, it would be interesting (at least for noncommutative resolutions of toric CY3s) to relate these algebras to the Yangians defined in terms of (generalised) McMahon modules and crystal melting in \cite{LY20,GY20}.
\smallbreak
For general $Q$ and $\mu$, the shifted BPS cohomology $\mathfrak{g}_{\tilde{Q},\tilde{W}^{\mu}}\coloneqq \DT_{\tilde{Q},\tilde{W}^{\mu}}[-1]$ carries a Lie algebra structure (by Theorem \ref{CIT}).  Even in the case $\mu=0$ we do not yet fully understand this Lie algebra for general $Q$.  We can at least try to relate the case of general $\mu$ to the case $\mu=0$, via the following construction:
\smallbreak
Fix $Q$ and $\mu$, and define
\begin{align*}
\mathfrak{g}_{\Even}\coloneqq &\bigoplus_{\dd\in\dvs \;\lvert\; \dd\cdot\mu=0}\mathfrak{g}_{\tilde{Q},\tilde{W},\dd}\\
\mathfrak{g}_{\Odd}\coloneqq &\bigoplus_{\dd\in\dvs \;\lvert\; \dd\cdot\mu\neq 0}\mathfrak{g}_{\tilde{Q},\tilde{W},\dd}.
\end{align*}
Then $\mathfrak{g}_{\Even}$ is a Lie subalgebra of $\mathfrak{g}_{\tilde{Q},\tilde{W}}$, and $\mathfrak{g}_{\Odd}$ is a Lie module for it, and we consider the extension 
\[
\mathfrak{g}^{\mu}=\mathfrak{g}_{\Even}\oplus \mathfrak{g}_{\Odd}[-1],
\]
with the Lie bracket 
\[
[(\alpha, \beta),(\alpha', \beta')]=([\alpha,\alpha'], [\alpha,\beta']-[\alpha',\beta]).
\]
\begin{question}
\label{BPSLAconj}
Is there is an isomorphism of Lie algebras 
\[
\mathfrak{g}^{\mu}\cong\mathfrak{g}_{\tilde{Q},\tilde{W}^{\mu}}?
\]
\end{question}
\begin{proposition}
\label{partial_a}
The answer to Question \ref{BPSLAconj} is yes in the case $\mu=0$, and for generic $\mu$.
\end{proposition}
\begin{proof}
The case $\mu=0$ is trivial, since then
\begin{align*}
\mathfrak{g}^{\mu}=&\mathfrak{g}_{\Even}\\
=&\mathfrak{g}_{\tilde{Q},\tilde{W}}
\end{align*}
and $\tilde{W}^{\mu}=\tilde{W}$.  In the case of generic $\mu$, Theorem \ref{main_thm} tells us that there is an isomorphism
\begin{equation}
\label{skac}
\mathfrak{g}_{\tilde{Q},\tilde{W}^{\mu}}\cong \mathfrak{g}_{\tilde{Q},\tilde{W}}[-1]
\end{equation}
as cohomologically graded vector spaces, and so it is sufficient to show that the Lie bracket on $\mathfrak{g}_{\tilde{Q},\tilde{W}^{\mu}}$ vanishes, since the Lie bracket on $\mathfrak{g}^{\mu}\coloneqq \mathfrak{g}_{\Odd}[-1]$ does by definition.  Combining \eqref{skac} and \eqref{Kacg}, we deduce that $\mathfrak{g}_{\tilde{Q},\tilde{W}^{\mu}}$ lies entirely in odd cohomological degree.  Since the Hall algebra multiplication, and hence the commutator Lie bracket, preserves cohomological degree, the Lie bracket vanishes as required.
\end{proof}
In this paper, we have ignored completely the question of whether there is a natural double to the Lie algebra $\mathfrak{g}_{\tilde{Q},\tilde{W}^{\mu}}$.  Despite Proposition \ref{partial_a} the expected answer to Question \ref{BPSLAconj}, once extended to the double of $\mathfrak{g}_{\tilde{Q},\tilde{W}^{\mu}}$ is \textit{no}; one of the motivations for this paper is work of Kevin Costello \cite{KMSRI}, in which he conjectures that in the case of e.g. the resolved conifold, the Lie bracket does not vanish on the fermionic part of the (doubled) BPS Lie algebra.  As we saw in \S \ref{ERC_sec}, and Example \ref{nccf} the (noncommutative) resolved conifold is a \textit{partial} fermionization of the QP $(\tilde{Q},\tilde{W})$ where $\tilde{Q}$ is the affine $A_1$ quiver; so the conifold is precisely the kind of case that Proposition \ref{partial_a} does not cover.
\subsection{Representation theory}
The boson-fermion correspondence is typically considered (by mathematicians) to be a part of representation theory, and so it would be remiss to finish the paper without saying \textit{anything} about representations of $\Coha_{\tilde{Q},\tilde{W}^{\mu}}$.  
\smallbreak
Fix a quiver $Q$, and a framing dimension vector $\ff\in \dvs$.  We form the quiver $Q_{\ff}$ by adding one extra vertex (labelled $\infty$) to $Q_0$, and $\ff_i$ arrows from $\infty$ to $i$ for each $i\in Q_0$.  Set $\mathscr{Q}=\widetilde{Q_{\ff}}$.  In other words, this is the usual doubled framed quiver that one uses to define Nakajima quiver varieties, but with an extra loop at every vertex (including a loop $\omega_{\infty}$ at the framing vertex).  
\smallbreak
For $\mu=\sum_{i\in Q_0} \mu_i e_i$ we define
\[
\tilde{W}_{\ff}^{\mu}=(\sum_{i\in \mathscr{Q}_0}\omega_i) (\sum_{a\in Q_{\ff}}[a,a^*])+\sum_{i\in Q_0}\mu_i\omega_i^2.
\]
\smallbreak
We write dimension vectors for $\mathscr{Q}$ as $\ee=(\dd,n)$ where $\dd\in\dvs$ and $n=\ee_{\infty}\in\mathbb{N}$.  We define
\[
\AA_{(\dd,1)}^{\stab}(\mathscr{Q})\subset \AA_{(\dd,1)}(\mathscr{Q})
\]
to be the subvariety containing those $\rho$ such that $\rho_{\infty}$ generates $\rho$ as a $\CC\mathscr{Q}$-module.  We call such $\CC\mathscr{Q}$-modules stable.  Then we define
\[
\mathscr{N}_{\ff,\dd}(Q)\coloneqq \AA_{(\dd,1)}^{\stab}(\mathscr{Q})/\Gl_{\dd}.
\]
This is a smooth variety, and is naturally isomorphic to the stack of stable $(\dd,1)$-dimensional $\CC\mathscr{Q}$-modules along with a trivialisation $\rho_{\infty}\cong\CC$.  We define
\[
\Flag_{(\dd,1),\dd'}(\mathscr{Q})
\]
to be the stack of pairs of a stable $(\dd,1)$-dimensional $\CC \mathscr{Q}$-module $\rho$, along with a $(\dd',0)$-dimensional submodule $\rho'\subset \rho$ and a trivialisation $\rho_{\infty}\cong\CC$.  Set $\dd''=\dd-\dd'$.  In the correspondence diagram
\[
\xymatrix{
\Mst_{\dd'}(\widetilde{Q})\times\mathscr{N}_{\ff,\dd''}(Q)&&\ar[ll]_-{\pi_1\times\pi_3}\Flag_{(\dd,1),\dd'}(\mathscr{Q})\ar[rr]^-{\pi_2}&&\mathscr{N}_{\ff,\dd}(Q)
}
\]
the morphism $\pi_2$ is proper, and so via push-forward and pull-back in vanishing cycle cohomology, we obtain an action of $\Coha_{\tilde{Q},\tilde{W}^{\mu}}$ on 
\[
\mathbf{M}^{\mu}_{\ff}(Q)\coloneqq\bigoplus_{\dd\in\dvs}\HO(\mathscr{N}_{\ff,\dd''}(Q),\phim{\Tr(\tilde{W}_{\ff}^{\mu})}\underline{\mathbb{Q}}_{\mathscr{N}_{\ff,\dd''}(Q)})\otimes \LLL^{\chi_{\mathscr{Q}}\left((1,\dd),(1,\dd)\right)/2}.
\]
Given a quiver $Q$ we denote by
\begin{align*}
\mu_{Q,\dd}\colon \mathbb{A}_{\dd}(\overline{Q})&\rightarrow \mathfrak{\gl}_{\dd}\\
\rho&\mapsto \sum_{a\in Q_1}[\rho(a),\rho(a^*)].
\end{align*}
the usual moment map.  We define
\[
Z(\ff,\dd)\subset \mathbb{A}^{\stab}_{(\dd,1)}(\overline{Q_{\ff}})
\]
the vanishing locus of the composition of $\mu_{Q_{\ff},(\dd,1)}$ with the natural projection $\mathfrak{\gl}_{(\dd,1)}\rightarrow \mathfrak{\gl}_{\dd}$.  Here the stability condition is the same as above: we restrict to the open locus of those $\rho$ such that $\rho_{\infty}$ generates.  Then the Nakajima quiver variety is defined to be the smooth variety
\[
X(\ff,\dd)=Z(\ff,\dd)/\Gl_{\dd}.
\]
\begin{proposition}
\label{NakSplit}
Denote by $g$ the restriction of the function $\Tr(\tilde{W}_{\ff}^{\mu})$ to $\mathscr{N}_{\ff,\dd}(Q)$, then
\[
\crit(g)\cong \begin{cases}X(\ff,\dd)&\textrm{if }\mu\cdot\dd\neq 0\\
X(\ff,\dd)\times\mathbb{A}^1&\textrm{if }\mu\cdot\dd=0.
\end{cases}
\]
Moreover the monodromic mixed Hodge module
\[
\phim{g}\underline{\mathbb{Q}}_{\mathscr{N}_{\ff,\dd}(Q)}\otimes \LLL^{\chi_{\mathscr{Q}}\left((\dd,1),(\dd,1)\right)/2}
\]
is analytically locally isomorphic to the constant mixed Hodge module on the critical locus of $g$.
\end{proposition}
\begin{proof}
The second part follows from the first: the holomorphic Bott--Morse lemma tells us that since the critical locus of $g$ is scheme-theoretically smooth, $g$ can be written analytically locally around $\crit(g)$ in the form
\[
g=x_1^2+\ldots+x_c^2
\]
where $c$ is the codimension of $\crit(g)$ inside $\mathcal{N}_{\ff,\dd}(Q)$.  The first part follows by the same argument as Proposition \ref{sspp}.
\end{proof}
In the special case $\mu=0$ it is possible to show that the monodromy of the rank one local system $\phin{g}\underline{\mathbb{Q}}_{\mathscr{N}_{\ff,\dd''}(Q)}$ is trivial \cite[Prop.6.3]{preproj3}, and we conjecture that this is always true.  Assuming the conjecture, Proposition \ref{NakSplit} implies that there are isomorphisms
\begin{equation}
\label{pfrep}
\mathbf{M}^{\mu}_{\ff}(Q)\cong \left(\bigoplus_{\substack{\dd\in\dvs\\ \mu\cdot\dd\neq 0}}\HO(X(\ff,\dd),\mathbb{Q})_{\vir}\right)\oplus \left(\bigoplus_{\substack{\dd\in\dvs\\ \mu\cdot\dd= 0}}\HO(X(\ff,\dd),\mathbb{Q})_{\vir}\otimes \LLL^{-1/2}\right).
\end{equation}
Note that the vacuum vector, spanning $\HO(X(\ff,0),\mathbb{Q})$, lies in the second summand.  
\smallbreak
Let $Q'$ be the full subquiver of $Q$ obtained by removing all vertices that support loops, as well as arrows to or from them, and let $\mathfrak{n}^-_{Q'}$ be the negative piece of the Kac--Moody Lie algebra associated to $Q'$.  By \cite[Thm.6.6]{preproj3} there is an inclusion of Lie algebras
\[
\mathfrak{n}^-_{Q'}\hookrightarrow \mathfrak{g}_{\tilde{Q},\tilde{W}}
\]
as the part of the BPS Lie algebra lying in cohomological degree zero, so that \eqref{pfrep} in case $\mu=0$ allows us to reconstruct (one half of) Nakajima's action of $\mathfrak{g}_{Q'}$ on the cohomology of quiver varieties \cite{Nak98}.  The conjecture suggests that the cohomology of Nakajima quiver varieties should be as crucial to the representation theory of partially fermionized BPS Lie algebras as they are to the representation theory of their bosonic counterparts.
\bibliographystyle{alpha}
\bibliography{Literatur}

\begin{thebibliography}{MMNS12}

\bibitem[Ach]{emhm}
P.~Achar.
\newblock {Equivariant mixed Hodge modules}.
\newblock {\em {Lecture notes to accompany talk ``Mixed Hodge modules and their
  applications''; available at author's website}}.

\bibitem[BBS13]{BBS}
K.~Behrend, J.~Byan, and B.~Szendr\H{o}i.
\newblock {Motivic degree zero Donaldson--Thomas invariants}.
\newblock {\em Invent. Math.}, 192(1):111--160, 2013.

\bibitem[BKL01]{BKL01}
J.~Bryan, S.~Katz, and C.~Leung.
\newblock {Multiple covers and the integrality conjecture for rational curves
  in Calabi--Yau threefolds}.
\newblock {\em J. Alg. Geom.}, 10(3):549--568, 2001.

\bibitem[CBdB04]{CBVdB04}
W.~Crawley-Boevey and M.~Van den Bergh.
\newblock {Absolutely indecomposable representations and Kac--Moody Lie
  algebras (with an appendix by Hiraku Nakajima}.
\newblock {\em Invent. Math.}, 155(3):537--559, 2004.

\bibitem[CBH98]{CBH98}
W.~Crawley-Boevey and M.~Holland.
\newblock {Noncommutative deformations of Kleinian singularities}.
\newblock {\em Duke Math. J.}, 92(3):605--635, 1998.

\bibitem[CKV01]{CKV}
F.~Cachazo, S.~Katz, and C.~Vafa.
\newblock {Geometric transitions and N=1 quiver theories}.
\newblock {\em arXiv preprint hep-th/0108120}, 2001.

\bibitem[Cos18]{KMSRI}
K.~Costello.
\newblock {The cohomological Hall algebra and M theory}.
\newblock https://www.msri.org/workshops/816/schedules/23868, 2018.
\newblock Talk at the MSRI workshop ``Structures in enumerative geometry''.

\bibitem[Dav17a]{Da13}
B.~Davison.
\newblock {The critical CoHA of a quiver with potential}.
\newblock {\em Quart. J. Math.}, 68(2):635--703, 2017.

\bibitem[Dav17b]{preproj}
B.~Davison.
\newblock {The integrality conjecture and the cohomology of preprojective
  stacks}.
\newblock {\url{https://arxiv.org/abs/1602.02110v3}}, 2017.

\bibitem[Dav20]{preproj3}
B.~Davison.
\newblock {BPS Lie algebras and the less perverse filtration on the
  preprojective CoHA}.
\newblock {\url{https://arxiv.org/abs/2007.03289}}, 2020.

\bibitem[DM20]{QEAs}
B.~Davison and S.~Meinhardt.
\newblock {Cohomological Donaldson--Thomas theory of a quiver with potential
  and quantum enveloping algebras}.
\newblock {\em Invent. Math.}, 221(3):777--871, 2020.

\bibitem[DOS20]{DOS20}
B.~Davison, J.~Ongaro, and B.~Szendr{\H{o}}i.
\newblock {Enumerating coloured partitions in 2 and 3 dimensions}.
\newblock {\em Math. Proc. Cambridge Philos. Soc.}, 169(3):479--505, 2020.

\bibitem[DP21]{DaPa20}
B.~Davison and T.~P\u{a}durariu.
\newblock Deformed dimensional reduction.
\newblock {\em Geom. Topol. (to appear)}, 2021.

\bibitem[DPS20]{DPS20}
D-E. Diaconescu, M.~Porta, and F.~Sala.
\newblock {McKay correspondence, cohomological Hall algebras and
  categorification}.
\newblock {\em arXiv preprint arXiv:2004.13685}, 2020.

\bibitem[Efi12]{Efimov}
A.~Efimov.
\newblock {Cohomological Hall algebra of a symmetric quiver}.
\newblock {\em Comp. Math.}, 148, no. 4:1133--1146, 2012.

\bibitem[ER06]{ER05}
P.~Etingof and E.~Rains.
\newblock {Central extensions of preprojective algebras, the quantum Heisenberg
  algebra, and 2-dimensional complex reflection groups}.
\newblock {\em J. Alg.}, 299(2):570--588, 2006.

\bibitem[Gin06]{ginz}
V.~Ginzburg.
\newblock {Calabi--Yau algebras}.
\newblock {\url{https://arxiv.org/abs/math/0612139}}, 2006.

\bibitem[GY20]{GY20}
D.~Galakhov and M.~Yamazaki.
\newblock {Quiver Yangian and Supersymmetric Quantum Mechanics}.
\newblock {\em arXiv preprint arXiv:2008.07006}, 2020.

\bibitem[Joy15]{Jo15}
D.~Joyce.
\newblock A classical model for derived critical loci.
\newblock {\em J. Diff. Geom.}, 101(2):289--367, 2015.

\bibitem[Kac80]{Kac80}
V.~Kac.
\newblock Infinite root systems, representations of graphs and invariant
  theory.
\newblock {\em {Invent. Math.}}, 56(1):57--92, 1980.

\bibitem[KM92]{KM92}
S.~Katz and D.~Morrison.
\newblock {Gorenstein threefold singularities with small resolutions via
  invariant theory for Weyl groups}.
\newblock {\em J. Alg. Geom.}, 1(3):449--530, 1992.

\bibitem[KS90]{KSsheaves}
M.~Kashiwara and P.~Schapira.
\newblock {\em {Sheaves on manifolds, volume 292 of Grundlehren der
  Mathematischen Wissenschaften [Fundamental Principles of Mathematical
  Sciences]}}.
\newblock Springer-Verlag, Berlin, 1990.

\bibitem[KS11]{KS2}
M.~Kontsevich and Y.~Soibelman.
\newblock {Cohomological Hall algebra, exponential Hodge structures and motivic
  Donaldson--Thomas invariants}.
\newblock {\em Commun. Number Theory Phys.}, 5, 2011.
\newblock arXiv:1006.2706.

\bibitem[KV00]{KV00}
M.~Kapranov and E.~Vasserot.
\newblock {Kleinian singularities, derived categories and Hall algebras}.
\newblock {\em Math. Ann.}, 316:565--576, 2000.

\bibitem[KV19]{KV19}
M.~Kapranov and E.~Vasserot.
\newblock {The cohomological Hall algebra of a surface and factorization
  cohomology}.
\newblock {\url{https://arxiv.org/abs/1901.07641}}, 2019.

\bibitem[KW98]{KW98}
I.~Klebanov and E.~Witten.
\newblock {Superconformal field theory on threebranes at a Calabi--Yau
  singularity}.
\newblock {\em Nucl. Phys. B}, 536(1-2):199--218, 1998.

\bibitem[LY20]{LY20}
W.~Li and M.~Yamazaki.
\newblock {Quiver Yangian from crystal melting}.
\newblock {\em J. High Energy Phys}, 2020(11):1--127, 2020.

\bibitem[Mas01]{Ma01}
D.~Massey.
\newblock {The Sebastiani--Thom isomorphism in the derived category}.
\newblock {\em Comp. Math.}, 125(3):353--362, 2001.

\bibitem[Mas16]{Ma09}
D.~Massey.
\newblock {Natural commuting of vanishing cycles and the Verdier dual}.
\newblock {\em Pac. J. Math.}, 284(2):431--437, 2016.

\bibitem[MMNS12]{MMNS}
A.~Morrison, S.~Mozgovoy, K.~Nagao, and B.~Szendr\H{o}i.
\newblock {Motivic Donaldson--Thomas invariants of the conifold and the refined
  topological vertex}.
\newblock {\em Adv. Math.}, 230, 2012.

\bibitem[MO19]{MO19}
D.~Maulik and A.~Okounkov.
\newblock Quantum groups and quantum cohomology.
\newblock {\em Ast{\'e}risque}, (408):1--+, 2019.

\bibitem[Moz11]{Moz11}
S.~Mozgovoy.
\newblock {Motivic Donaldson-Thomas invariants and McKay correspondence}.
\newblock {\url{https://arxiv.org/abs/1107.6044}}, 2011.

\bibitem[MR19]{Meinhardt14}
S.~Meinhardt and M.~Reineke.
\newblock {Donaldson--Thomas invariants versus intersection cohomology of
  quiver moduli}.
\newblock {\em J. f\"ur die Reine und Angew. Math. (Crelles Journal)},
  2019(754):143--178, 2019.

\bibitem[Nak98]{Nak98}
H.~Nakajima.
\newblock {Quiver varieties and Kac--Moody algebras}.
\newblock {\em Duke Math. J.}, 91(3):515--560, 1998.

\bibitem[Nak01]{Nak01}
H.~Nakajima.
\newblock Quiver varieties and finite dimensional representations of quantum
  affine algebras.
\newblock {\em J. Am. Math. Soc}, pages 145--238, 2001.

\bibitem[RS17]{RS17}
J.~Ren and Y.~Soibelman.
\newblock {Cohomological Hall Algebras, Semicanonical Bases and
  Donaldson--Thomas Invariants for 2-dimensional Calabi--Yau Categories (with
  an Appendix by Ben Davison)}.
\newblock In {\em Algebra, Geometry, and Physics in the 21st Century}, pages
  261--293. Springer, 2017.

\bibitem[RSYZ20]{RSYZ20}
M.~Rapcak, Y.~Soibelman, Y.~Yang, and G.~Zhao.
\newblock {Cohomological Hall algebras and perverse coherent sheaves on toric
  Calabi-Yau 3-folds}.
\newblock {\em arXiv preprint arXiv:2007.13365}, 2020.

\bibitem[Sai]{Saito10}
M.~Saito.
\newblock {Thom--Sebastiani Theorem for Hodge Modules}.
\newblock {\em preprint 2010}.

\bibitem[SV13]{ScVa13}
O.~Schiffmann and E.~Vasserot.
\newblock {Cherednik algebras, W-algebras and the equivariant cohomology of the
  moduli space of instantons on $\mathbb{A}^2$}.
\newblock {\em Pub. Math IH{\'E}S}, 118(1):213--342, 2013.

\bibitem[Sze08a]{conifold}
B.~Szendr\H{o}i.
\newblock {Non-commutative Donaldson-Thomas theory and the conifold}.
\newblock {\em Geom. Topol}, 12(2):1171--1202, 2008.

\bibitem[Sze08b]{Sz08}
B.~Szendr{\H{o}}i.
\newblock Sheaves on fibered threefolds and quiver sheaves.
\newblock {\em Commun. Math. Phys.}, 278(3):627--641, 2008.

\bibitem[Var00]{Var00}
M.~Varagnolo.
\newblock {Quiver varieties and Yangians}.
\newblock {\em Lett. Math. Phys.}, 53(4):273--283, 2000.

\bibitem[VB10]{Vel10}
A.~Quintero V{\'e}lez and A.~Boer.
\newblock {Noncommutative resolutions of ADE fibered Calabi-Yau threefolds}.
\newblock {\em Comm. Math. Phys.}, 297(3):597--619, 2010.

\bibitem[vdB04]{NonComCrep}
M.~van~den Bergh.
\newblock Non-commutative crepant resolutions.
\newblock In {\em The legacy of {N}iels {H}enrik {A}bel}, pages 749--770.
  Springer, Berlin, 2004.

\bibitem[YZ16]{YZ16}
Y.~Yang and G.~Zhao.
\newblock {On two cohomological Hall algebras}.
\newblock {\em Proc. Roy. Soc. Edinburgh A}, pages 1--27, 2016.

\bibitem[YZ18]{YZ18}
Y.~Yang and G.~Zhao.
\newblock {The cohomological Hall algebra of a preprojective algebra}.
\newblock {\em Proc. London Math. Soc.}, 116(5):1029--1074, 2018.

\end{thebibliography}

\end{document}